\documentclass[12pt,a4paper,reqno]{amsart}
\usepackage{amsmath,amsfonts,amssymb,amsthm,amsbsy,euscript}
\usepackage{mathtools} % coloneqq
\usepackage[small,nohug,heads=LaTeX]{diagrams}
\diagramstyle[labelstyle=\scriptstyle]

\def\oldvec{\mathaccent "017E\relax }

\newcommand{\by}[1]{\textit{{#1}}}
\newcommand{\jour}[1]{\textit{{#1}}}
\newcommand{\vol}[1]{\textbf{{#1}}}
\newcommand{\book}[1]{\textrm{{#1}}}

\title[The orientation morphism]{The orientation morphism: from graph cocycles to deformations of Poisson structures}

\author[R.~Buring]{R.~Buring${}^{\ddagger}$}
\thanks{${}^{\ddagger}$\textit{Address}: %Algebraic Geometry, Topology \& Number Theory,
Institut f\"ur Mathematik, %FB 08 -- Physik, Mathematik und Informatik,
Johannes Gutenberg\/--\/Uni\-ver\-si\-t\"at,
Staudingerweg~9, %4.OG,
\mbox{D-\/55128} Mainz, Germany.
\quad %${}^{\ddagger}$\:
{E-mail}: \texttt{rburing\symbol{"40}uni-mainz.de}.}

\author[A.~V.~Kiselev]{A.\,V.\,Kiselev${}^{\S}$}
\thanks{${}^{\S}$\textit{Address}: %Johann 
Bernoulli Institute for Mathematics, %\& %and
Computer Science \& %and 
Artificial Intelligence, University of Groningen, P.O.~Box~407, 9700~AK Groningen, The~Netherlands. %\quad %${}^{\S}$\:
{E-mail}: \texttt{A.V.Kiselev\symbol{"40}rug.nl}.%\quad%
}

\date{2 December 2018}


\newtheorem{thm}{Theorem}

\newtheorem{prop}{Proposition}
\newtheorem{cor}{Corollary}
\theoremstyle{remark}
\newtheorem{rem}{Remark}
\newtheorem{example}{Example}
\theoremstyle{definition}
\newtheorem*{notation}{Notation}
\newtheorem{define}{Definition}

\newcommand{\schouten}[1]{[\![#1]\!]}
\newcommand{\NR}{\text{\textup{NR}}}
\newcommand{\Or}{{\rm O\oldvec{r}}}
\newcommand{\littleOr}{{\rm \oldvec{or}}}
\newcommand{\cP}{\mathcal{P}}

\newcommand{\cQ}{\mathcal{Q}}
\newcommand{\cX}{{\EuScript X}}    %{\mathcal{X}}

\newcommand{\stick}{\bullet\!\!\!-\!\!\!\bullet}
\newcommand{\decoratedwedge}{\xleftarrow[\text{\ Left}]{i}\!\! \bullet\!\! \xrightarrow[\text{Right}]{j}}
\DeclareMathOperator{\Id}{d}

\DeclareMathOperator{\End}{End}
\newcommand{\Tpoly}{T_{\textrm{\normalfont \textup{poly}}}}
\newcommand{\insertion}[2]{#1\mathbin{ \vec{\circ} }#2}
\newcommand{\Gra}[1]{\big(\!\operatorname{Gra}_{\text{\textup{\#Vert}}\eqqcolon #1 \geqslant 1}^{\bigwedge_i \text{\textup{edge}}_i}\!\big)_{S_#1}}

\newcommand{\cev}[1]{\reflectbox{\ensuremath{\vec{\reflectbox{\ensuremath{#1}}}}}}

\begin{document}

\begin{abstract}
\noindent
We recall the construction of the Kontsevich graph orientation morphism $\gamma \mapsto \Or(\gamma)$ which maps cocycles $\gamma$ in the non-oriented graph complex to infinitesimal symmetries $\dot{\cP} = \Or(\gamma)(\cP)$ of Poisson bi-vectors on affine manifolds.
We reveal in particular why there always exists a factorization of the Poisson cocycle condition $\schouten{\cP,\Or(\gamma)(\cP)} \doteq 0$ through the differential consequences of the Jacobi identity $\schouten{\cP,\cP}=0$ for Poisson bi-vectors $\cP$.
To illustrate the reasoning, we use the Kontsevich tetrahedral flow $\dot{\cP} = \Or(\boldsymbol{\gamma}_3)(\cP)$, as well as the flow produced from the Kontsevich--Willwacher pentagon-wheel cocycle $\boldsymbol{\gamma}_5$ and the new flow obtained from the heptagon-wheel cocycle~$\boldsymbol{\gamma}_7$ in the unoriented graph complex.
\end{abstract}

\maketitle

\enlargethispage{0.3\baselineskip}
\noindent
\textbf{Introduction.}\ 
On an affine manifold $%\c
M^r %d
$, the Poisson bi-vector fields are those satisfying the Jacobi identity $\schouten{\cP,\cP} = 0$, where $\schouten{\cdot,\cdot}$ is the Schouten bracket (\cite{LaurentGengouxPichereauVanhaecke}, see also Example~\ref{ExSchouten} below).
A deformation $\cP \mapsto \cP + \varepsilon \cQ + \bar{o}(\varepsilon)$ of a Poisson bi-vector $\cP$ preserves the Jacobi identity infinitesimally if $\schouten{\cP, \cQ} = 0$.
If, by assumption, the deformation term $\cQ$ (itself not necessarily Poisson) depends on the bi-vector $\cP$, then the equation $\schouten{\cP, \cQ(\cP)} \doteq 0$ must be satisfied by force of $\schouten{\cP,\cP} = 0$.
In \cite{Ascona96} %M. 
Kontsevich designed a way to produce infinitesimal deformations $\dot{\cP} = \cQ(\cP)$ which are \emph{universal} with respect to all Poisson structures on all affine manifolds: for a given bi-vector $\cP$, the coefficients of bi-vector $\cQ(\cP)$ are differential polynomial in the coefficients of $\cP$.

The original construction from \cite{Ascona96} goes % works
in three steps, as follows.
First, recall that the vector space $\Gra{n}$ of unoriented finite graphs with unlabelled vertices and wedge ordering on the set of edges carries the structure of a complex with respect to the vertex-expanding differential~$\Id$.
In fact, this space is a differential graded Lie algebra such that the differential $\Id$ is the Lie bracket with a single edge, $\Id = [\stick, \cdot]$.
Let $\gamma = \sum_i c^i \gamma_i$ be a sum of graphs with $n$ vertices and $2n-2$ edges, satisfying $\Id(\gamma) = 0$.
Then let us sum --\,with signs, which will be discussed in~\S\ref{SecOr} below\,-- over all possible ways to orient the graphs $\gamma_i$ in the cocycle $\gamma$ such that each vertex is the arrowtail for two outgoing edges; %one 
create %s 
two extra edges going to two new vertices, the sinks.
% TODO: explain, ordering of sinks
Secondly, skew-symmetrize (w.r.t.\ %with respect to 
the sinks) the resulting sum of Kontsevich oriented graphs.
Finally, insert a Poisson bi-vector $\cP$ into each vertex of every $\gamma_i$ in the sum of Kontsevich graphs at hand.
%Under the Kontsevich orientation morphism, $\dd$-cocycle graphs $\gamma$ on $n$ vertices and $2n-2$ edges yield infinitesimal symmetries $\dot{\cP} = \Or(\gamma)(\cP)$ of Poisson bivectors $\cP$ on finite-dimensional affine manifolds $\cM^r$. % RB: \text{d} is also the differential
% Namely,
Now, every oriented graph built of the decorated wedges $\decoratedwedge$ determines a differential-polynomial expression in the coefficients $\cP^{ij}(x^1$,\ $\ldots$,\ $x^r)$ of a bivector $\cP$ whenever the arrows $\xrightarrow{a}$ denote derivatives $\partial/\partial x^a$ in a local coordinate chart% on $\cM^r$
, each vertex $\bullet$ at the top of a wedge contains a copy of $\cP$, and one takes the product of vertex contents and sums up over all the indexes.
The right-hand side of the symmetry flow $\dot{\cP} = \cQ(\cP)$ is obtained!

We give an explicit, relatively elementary proof that this recipe % scheme, construction
does the job% works
, i.e. why the Poisson cocycle condition $\schouten{\cP, \cQ(\cP)} \doteq 0$ is satisfied for % verified by
every Poisson structure $\cP$, and for every $\cQ = \Or(\gamma)$ obtained from a graph cocycle $\gamma \in \ker \Id$ in this way.
The % Our
reasoning is based on that given by Jost \cite{Jost2013}, which in turn follows an outline by Willwacher \cite{WillwacherGRT}, itself referring to the seminal paper \cite{Ascona96} by Kontsevich.

At the same time, the present text % paper
concludes a series of papers \cite{tetra16,f16,sqs17} with an empiric search for % constructions of 
the factorizations $\schouten{\cP, \cQ(\cP)} %\doteq 0
= \Diamond\bigl(\cP,\schouten{\cP,\cP}\bigr)$ using the Jacobiator $\schouten{\cP,\cP} %= 0
$, as well as containing an independent verification of the numerous rules of signs for %the 
many graded objects under study --- the ultimate % overall
aim being to %the 
understand %ing of 
the %orientation 
morphism~$\Or$.

Section~\ref{SecOr} %{SecEndo} %{SecGraphs} 
establishes the formula\footnote{The existence of this formula with \emph{some} vanishing right-hand
side is implied in~\cite{Ascona96, WillwacherGRT,DolgushevRogersWillwacher} where it is stated that there is an action of the graph complex on Poisson structures (or Maurer--Cartan elements of $\Tpoly(M)$).
The precise right-hand side is all but written in~\cite{Jost2013}; still to the best of our knowledge, the exact formula is presented %written
   here and
on p.~\pageref{EqDiamond} below
for the first time. --- The same applies to Jacobi identity~\eqref{EqJacNR} for the Lie bracket \emph{of graphs} (cf.~\cite{Identities32}).%
%%%%%%%%%%%
% Ascona96: 
% WillwacherGRT: "action on T_poly yields an action on the space of its MC elements"
% DolgushevRogersWillwacher: "Action of GC on T_poly(P)"
% They imply the existence of _some_ right-hand side which vanishes, but not
} of Poisson cocycle factorization through the Jacobiator $\schouten{\cP,\cP} %= 0
$:
\begin{multline}\label{EqOrCocycle}
2\cdot \schouten{\cP, \Or(\gamma)(\cP, \ldots, \cP)} 
    %&= \Or(\gamma)(\cP, \ldots, \cP, \schouten{\cP,\cP}) %\nonumber 
= \Or(\gamma)(\schouten{\cP,\cP}, \cP, \ldots,\cP) + \ldots + {} \\
   %& \hskip 1em {} + 
+ \Or(\gamma)(\cP,\ldots,\cP,\schouten{\cP,\cP},\cP,\ldots,\cP) + \ldots + %{} \nonumber \\
%& \hskip 1em {} + 
\Or(\gamma)(\cP,\ldots,\cP,\schouten{\cP,\cP}), %\nonumber
\end{multline}
% TODO: signs?
where the r.-h.s.\ %right-hand side 
consists of oriented graphs with one copy of the tri-vector $\schouten{\cP,\cP}$ inserted consecutively into a % one
vertex of the graph(s)~$\gamma$.

We illustrate the work of % RB: phrasing
orientation morphism $\Or$ %\colon 
which maps $\ker\,\Id \ni \gamma \mapsto \cQ(\cP) \in \ker \schouten{\cP, \cdot}$ 
    % RB: domain and codomain correct?
by using four % TODO: update number
examples, which include in particular the first % starting
elements $\boldsymbol{\gamma}_3$,\ $\boldsymbol{\gamma}_5$,\ $\boldsymbol{\gamma}_7 \in \ker \Id$ of nontrivial graph cocycles found by Willwacher in~\cite{WillwacherGRT}: the Kontsevich tetrahedral flow $\dot{\cP} = \Or(\boldsymbol{\gamma}_3)(\cP)$ (see~\cite{Ascona96} and~\cite{tetra16, f16}), the Kontsevich--Willwacher pentagon wheel cocycle $\boldsymbol{\gamma}_5$ and the respective flow $\dot{\cP} = \Or(\boldsymbol{\gamma}_5)(\cP)$ (here, see \cite{JNMP17} and \cite{WillwacherGRT}), and similarly, the heptagon-wheel cocycle $\boldsymbol{\gamma}_7$ and its flow.
In each case, %the logic of 
the reasoning % RB: phrasing
reveals a factorization $\schouten{\cP, \Or(\gamma)(\cP)} = \Diamond(\cP, \schouten{\cP,\cP})$ through the Jacobi identity $\schouten{\cP,\cP} = 0$.
For the tetrahedral flow $\dot{\cP} = \Or(\boldsymbol{\gamma}_3)(\cP)$ we thus
recover the factorization of $\schouten{\cP, \dot{\cP}}$ -- in terms of the ``Leibniz'' graphs with the tri-vector $\schouten{\cP,\cP}$ inside -- which had been obtained in \cite{f16} by a brute force calculation.
Let it be noted that such factorizations, $\schouten{\cP, \dot{\cP}} = \Diamond(\cP, \schouten{\cP,\cP})$, are known to be non-unique for a given flow $\dot{\cP}$; the scheme which we presently consider provides % yields
one such operator $\Diamond$ (out of many, possibly).
% TODO: explain non-uniqueness

Trivial graph cocycles, i.e.\ $\Id$-coboundaries $\gamma = \Id(\beta)$ also serve as an illustration.
Under the orientation mapping $\Or$ their ``potentials'' $\beta$ (sums of graphs with $n-1$ vertices and $2n-3$ edges) are transformed into the vector fields~$\cX$, also codified % encoded
by the Kontsevich oriented graphs, which trivialize the respective flows $\dot{\cP} = \Or(\gamma)(\cP)$ in the space of bi-vectors: namely, $\Or(\Id(\beta))(\cP) = \schouten{\cP, \Or(\beta)(\cP)}$ so that the resulting flow $\dot{\cP} = \cQ(\cP) = \schouten{\cP, \cX(\cP)}$ is trivial in Poisson cohomology. % TODO: constant?
%Here 
We offer 
%as 
an example on p.~\pageref{ExBeta6}: here, $\cX(\cP) = 2\Or(\boldsymbol{\beta}_6)(\cP)$. % TODO: \beta

This paper continues in~\S\ref{SecExamples} with some statistics about the number of graphs 
(\textit{i})  in the ``known'' cocycles $\gamma = \sum c^i \gamma_i \in \ker \Id$, 
(\textit{ii})  in the respective flows $\cQ = \Or(\gamma)$ which consist of the oriented Kontsevich graphs, 
(\textit{iii})  in the factorizing operators $\Diamond$ (provided by the proof) which are encoded by the Leibniz  graphs (see \cite{sqs15, f16}), and 
(\textit{iv}) in the cocycle equations $\schouten{\cP, \Or(\gamma)(\cP)} \doteq 0$.
We see that for thousands and millions of oriented graphs in the left- and right-hand sides of \eqref{EqOrCocycle} the coefficients match perfectly.

% Remark: no flows are known other than those obtained by orienting a cocycle \gamma
% NB: cocycle \gamma is not necessarily connected

\section{The parallel worlds of graphs and endomorphisms}% operadically, of the space of poly/multivector fields
\label{SecGraphs}
\noindent%
The universal deformations $\dot{\cP} = \cQ(\cP)$ which we consider will be given by certain endomorphisms evaluated at copies of a given Poisson structure $\cP$.
In particular, the resulting expressions will be differential polynomials in the coefficients of $\cP$.
Moreover, such expressions will be built using graphs, so that properties of objects in the graph complex are translated into properties of the objects realized by the graphs in the Poisson complex.
To this end, let us recall and compare the notions of operads of non-oriented graphs and of endomorphisms of multi%poly
-vector fields on affine manifolds.
This material is standard; we follow~\cite{Ascona96, Jost2013, WillwacherGRT, MerkulovWillwacherGRTBV}.

\subsection{Endomorphisms $\End(\Tpoly(%\c
M)[1])$ (e.g., the Schouten bracket~$\schouten{\cdot,\cdot}$)}\label{SecEndo}
\noindent%
Denote the shifted-graded vector space of all multi-vector fields on the manifold~$%\c
M^r$ % at hand
by\footnote{This notation for the space of multi-vectors should not be confused with a similar notation for the space of vector fields with polynomial coefficients on an affine manifold $%\c
M^r$. Nor should it be read as the space of multi-vectors on a super-manifold.} % RB: bundle?
\[ \Tpoly(%\c
M)[1] = \bigoplus_{\bar{\ell} \geqslant -1} \Tpoly^{\ell}(%\c
M) \qquad \text{ where } \quad \ell = \bar{\ell} + 1. \]%%%
The grading in $\Tpoly(%\c
M)[1] = T^{\downarrow[1]}_{\text{\textup{poly}}}(M)$
is shifted \emph{down} so %such 
that, by definition, a bi-vector~$\cP$ has degree $|\cP|=2$ but $\bar{\cP}=1$,~etc.
   %$\bar{2} = 1$,   %NB: "\bar" is defined NOT on \BBZ.
We let the multi\/-\/vectors be %(Those are 
encoded in a standard way using a local coordinate chart $x^1$,\ $\ldots$,\ $x^r$ on $%\c
M^r$ and the respective parity\/-\/odd variables $\xi_1$,\ $\ldots$,\ $\xi_r$ along the reverse\/-\/parity fibres of $\Pi T^*%\c
M^r$ over that chart.
For example, a bi-vector is written in coordinates as $\cP=\sum_{1\leqslant i<j\leqslant r} P^{ij}(\boldsymbol{x}) \xi_i \xi_j$.\footnote{Our notation is such that the wedge product of multi-vectors does not include any constant factor.}

An endomorphism of $\Tpoly(%\c
M)[1]$ of arity $k$ and degree~$\bar{d}$ is a $k$-linear (over the field~$\mathbb{R}$) map $\theta\colon \Tpoly(%\c
M)[1] \otimes \ldots \otimes \Tpoly(%\c
M)[1] \to \Tpoly(%\c
M)[1]$, not necessarily (graded-)\/skew in its $k$ arguments, and such that for grading-homogeneous arguments we have that 
\[\theta\colon \Tpoly^{\bar{d}_1}(%\c
M) \otimes \ldots \otimes \Tpoly^{\bar{d}_k}(%\c
M) \to \Tpoly^{\bar{d}_1+\ldots+\bar{d}_k+\bar{d}}(%\c
M),\]
i.e.\ $\theta$ restricts to a map of degree~$\bar{d}$. % TODO: phrasing

\begin{example}
\label{ExSchouten}
The Schouten bracket $\schouten{\cdot, \cdot} \colon \Tpoly^{\bar{d}_1}(%\c
M) \otimes \Tpoly^{\bar{d}_2}(%\c
M) \to \Tpoly^{\bar{d}_1+\bar{d}_2}(%\c
M)$ has arity $2$ and shifted degree $\overline{\deg}\,(\schouten{\cdot,\cdot})=0$
(note $\bigl| \schouten{\cdot,\cdot} \bigr|=-1$).
   %%\bar{1} = 0$. % TODO: \pi_S?
It is expressed %given 
in coordinates by the formula
\[ 
\schouten{\cP,\cQ} = \sum_{\ell=1}^{r} (\cP) \frac{\cev{\partial}}{\partial\xi_\ell} \cdot
\frac{\vec{\partial}}{\partial x^\ell}(\cQ) -
(\cP) \frac{\cev{\partial}}{\partial x^\ell}  \cdot \frac{\vec{\partial}}{\partial \xi_\ell}(\cQ).
\]
\end{example}

\begin{notation}%\noindent
The bi-graded vector space of %such 
endomorphisms under study is denoted by 
\[\End\bigl(\Tpoly(%\c
M)[1]\bigr) = \bigoplus_{\substack{\bar{d}\in\mathbb{Z}\\k \geqslant 1}} 
\End^{k,\bar{d}}\bigl(\Tpoly(%\c
M)[1]\bigr).\]
This space has the structure of an operad (with an action by the permutation group~$S_k$ on the part of arity~$k$): indeed, endomorphisms can be inserted one into another.

Let $\theta_a$ and $\theta_b$ be two endomorphisms of respective arities~$k_a$ and~$k_b$.
The insertion of $\theta_a$ into the $i^{\text{th}}$ argument of $\theta_b$ is denoted by $\theta_a \mathbin{\vec{\circ}_i} \theta_b$.
For instance, $(\theta_a \mathbin{\vec{\circ}_1} \theta_b)(p_1,\ldots,p_{k_a+k_b-1}) = \theta_b(\theta_a(p_1$,\ $\ldots$,\ $p_{k_a})$,\ $p_{k_a+1}$,\ $\ldots$,\ $p_{k_a+k_b-1})$.
Likewise, the notation $\theta_a \mathbin{\cev{\circ}_i} \theta_b$ means the insertion of the succeeding object $\theta_b$ into the preceding $\theta_a$, whence $(\theta_a \mathbin{\cev{\circ}_1} \theta_b)(\boldsymbol{p}) = \theta_a(\theta_b(p_1$,\ $\ldots$,\ $p_{k_b})$,\ $p_{k_b + 1}$,\ $\ldots$,\ $p_{k_a + k_b - 1})$.
Without an arrow pointing left, this notation~$\cev{\circ}_i$ is used in other papers; it is also natural because the graded objects $\theta_a$ and $\theta_b$ are not swapped.
\end{notation}

\begin{define}
The \emph{insertion} $\vec{\circ}$ of an endomorphism $\theta_a$ into an endomorphism $\theta_b$ of arity $k_b$ is the sum of insertions: $\insertion{\theta_a}{\theta_b} = \sum_{i=1}^{k_b} \theta_a \mathbin{\vec{\circ}_i} \theta_b$.
The graded commutator of endomorphisms of degrees $d_a$ and $d_b$ is $[\theta_a, \theta_b] = \insertion{\theta_a }{\theta_b} - (-)^{
   %\bar{d}_a \bar{d}_b  %%% NO. In earnest, without shift.
|\theta_a|\cdot|\theta_b| } \insertion{\theta_b}{\theta_a}$.
An endomorphism $\theta$ of arity~$k$ is \emph{skew} with respect to permutations of its graded arguments if it acquires the Koszul sign, $\theta (p_1,\ldots,p_k) = \epsilon_{\boldsymbol{p}}(\sigma)\theta(p_{\sigma(1)},\ldots,p_{\sigma(k)})$ under $\sigma \in S_k$. % TODO RB: sign depends not really on p but on degrees.
Here $\epsilon_{\boldsymbol{p}}((1\ 2)) = (-)^{(1\ 2)} (-)^{
   %|p_1|\cdot|p_2|  %%% NO. In earnest, with shift.
\bar{p}_1\cdot\bar{p}_2 }$ and similarly for all other transpositions which generate the permutation group~$S_k$.
Suppose that both of the endomorphisms $\theta_a$ and $\theta_b$ from the above are graded skew-symmetric.
The \emph{Nijenhuis--Richardson bracket} $[\theta_a,\theta_b]_{\NR}$ of those skew endomorphisms (of degrees $d_a$ and $d_b$ respectively) is the skew-symmetrization of $[\theta_a, \theta_b]$ with respect to the permutations, graded by the Koszul signs.
\end{define}

\begin{example}
The shifted-graded skew-symmetric Schouten bracket \[ \pi_S(p_1,p_2) \coloneqq (-)^{|{p_1}|-1}\schouten{p_1, p_2} \in \End^2(\Tpoly(%\c
M)[1]) \]
 %factor between \pi_s and \schouten{}
of multivectors $a,b,c$ of respective homogeneities satisfies the shifted-graded Jacobi identity \[\schouten{a,\schouten{b,c}} - (-)^{\bar{a}\bar{b}}\schouten{b,\schouten{a,c}} = \schouten{\schouten{a,b},c} = 0,\] 
or equivalently, \[ \schouten{a,\schouten{b,c}} + (-)^{\bar{a}\bar{b}+\bar{a}\bar{c}}\schouten{b,\schouten{c,a}} + (-)^{\bar{c}\bar{a} + \bar{c}\bar{b}}\schouten{c,\schouten{a,b}} = 0. \]
Taken four times, $[\pi_S, \pi_S]_{\NR}$ evaluated (with Koszul signs shifted by $\deg \schouten{\cdot,\cdot} = -1$) at $a$,\ $b$,\ $c$ yields the l.-h.s.\ %left-hand side 
of the Jacobi identity for $\schouten{\cdot, \cdot}$.
% We thus have
This shows that $[\pi_S,\pi_S]_{\NR} = 0$.
\end{example}

\begin{prop}%{claim}
The Nijenhuis--Richardson bracket \textup{(}of homogeneous arguments of respective degrees\textup{)} itself satisfies the graded Jacobi identity 
\begin{equation}\label{EqJacNR} 
[a,[b,c]_{\NR}]_{\NR} - (-)^{|a|\cdot|b|}[b,[a,c]_{\NR}]_{\NR} = [[a,b]_{\NR},c]_{\NR}, 
\end{equation}
or equivalently, \[ [a,[b,c]_{\NR}]_{\NR} + (-)^{|a|\cdot|b| + |a|\cdot|c|}[b,[c,a]_{\NR}]_{\NR} + (-)^{|c|\cdot|a|+|c|\cdot|b|}[c,[a,b]_{\NR}]_{\NR} = 0. \]
\end{prop}%{claim}

\begin{cor}
The map %operation 
$\partial \coloneqq [\pi_S, \cdot]_{\NR}$ is a differential on the space of skew endo\-mor\-ph\-isms.
\end{cor}

% TODO? Remark: Jacobi upon Jacobi upon Jacobi.

\subsection{Graphs vs endomorphisms}% and
\label{SecOr}
\noindent%
Having studied the natural differential graded Lie algebra (dgLa) structure on the space of graded skew-symmetric endomorphisms $\End_{\text{skew}}^{*,*}(\Tpoly(%\c
M)[1])$, we observe that its construction goes in parallel with the dgLa structure on the vector space $\bigoplus_k \Gra{k}$ of finite non-oriented graphs with wedge ordering of edges
   %and unlabelled vertices
(and without leaves). 
   % RB: need all numbers of vertices here
Referring to \cite{DolgushevRogersWillwacher, Jost2013, Ascona96, WillwacherGRT} (and references therein), as well as to \cite{JNMP17,JPCS17,Identities32} with explicit examples of calculations in the graph complex, % TODO: sqs17?
we summarize % sum up, frame
the set of analogous objects and structures in Table~\ref{TabParallel} below.
%%%
\begin{table}[htb]
\caption{From graphs to endomorphisms: the respective objects or structures.}\label{TabParallel}
\vskip 0.5em
\begin{tabular}{p{0.49\textwidth} p{0.49\textwidth}}
%\underline
{\textbf{World of graphs}} & %\underline
{\textbf{World of endomorphisms}} \\
\hline
Graphs $(\gamma, E(\gamma))$ & Endomorphisms \\
Insertion $\vec{\circ}_i$ of graph into $i^{\text{th}}$ vertex & Insertion of endomorphism into $i^{\text{th}}$ argument \\
Insertion $\insertion{}{}$ of graph into graph & Insertion $\insertion{}{}$\\
Bracket $[a,b] = \insertion{a}{b}-(-)^{|E(a)|\cdot|E(b)|} \insertion{b}{a}$ & Bracket $[a,b] = \insertion{a}{b}-(-)^{|a|\cdot|b|} \insertion{b}{a}$ \\ % TODO: shift, overline
Lie bracket $([a,b], E([a,b]) \coloneqq E(a) \wedge E(b))$ & Nijenhuis-Richardson bracket $[a,b]_{\NR}$ on the space of skew endomorphisms \\ % TODO: notation?
The stick $\stick$ & The Schouten bracket $\pi_S = \pm\,\schouten{\cdot, \cdot}$ \\
Master equation $[\stick,\stick] = 0$ & Master equation $[\pi_S, \pi_S]_{\NR} = 0$ \\
Graded Jacobi identity for $[\cdot,\cdot]$ & Graded Jacobi identity for $[\cdot,\cdot]_{\NR}$ \\
Differential $\Id = [\stick,\cdot]$ & Differential $\partial = [\pi_S, \cdot]_{\NR}$
\end{tabular}
\end{table}

The orientation morphism $\Or$, which we presently discuss, provides a transition ``\!\!$\implies$\!\!'' from graphs to endomorphisms.
%%%
Our goal is to have a % the
Lie algebra morphism
\[
%\text{\raisebox{0.5mm}[8mm][5mm]{$%
\Bigl\{
%\text{\parbox{33mm}{\footnotesize%
%$\textstyle
\bigoplus_k \Gra{k},\quad
\Id = [\stick, \cdot]
%}
%} 
\Bigr\}%
\xrightarrow{\ \ \Or\ \ }%
\Bigl\{
%\text{\parbox{34mm}{\footnotesize%
%$
\End_{\text{skew}}^{*,*}(\Tpoly(%\c
M)[1]),\quad %$\\
\partial = [\pi_S, \cdot]_{\NR}%$
%}
%}
\Bigr\}%
%$}%}
,
\]
hence a dgLa morphism because the differentials $\Id = [\stick,\cdot]$ and $\partial = [\pi_S, \cdot]_{\NR}$ are the adjoint actions of the Maurer--Cartan elements.

In the meantime, we claim without proof that the edge $\stick$ is taken to the Schouten bracket $\pi_S = \pm\,\schouten{\cdot,\cdot}$ by $\Or$: namely, $\stick \mapsto \pi_S=\pm\,\schouten{\cdot,\cdot}$ (see~\eqref{EqSchoutenPi} below).
So, having a Lie algebra morphism implies that $\Or([\stick, \gamma]) = [\pi_S, \Or(\gamma)]_{\NR}$ for a graph $\gamma$ with edge ordering $E(\gamma)$, i.e.\ the following diagram is commutative:
\begin{diagram}
(\gamma,E(\gamma))     &\rMapsto^{\ \ \ \ \ \Or} & \Or(\gamma)\\
\dMapsto_{\Id}        & & \dMapsto_{\partial} \\
[\stick,\gamma]     & \rMapsto^{\!\!\Or} & [\pi_S, \Or(\gamma)]_{\NR}.
\end{diagram}
%\[\begin{CD}
%(\gamma,E(\gamma))     @>\Or>>  \Or(\gamma)\\
%@VV{\Id}V        @VV{\partial}V\\
%[\stick,\gamma]     @>\Or>>  [\pi_S, \Or(\gamma)]_{\NR}.
%\end{CD}\]
%%%
When this diagram is reached, % justified, worked out
it will be seen -- by evaluating the endomorphisms at copies of $\cP$ --
why the mapping of $\Id$-cocycles in the graph complex to Poisson cocycles $\in \ker\, \schouten{\cP,\cdot}$ is well defined.
This will solve the problem of producing universal infinitesimal symmetries $\dot{\cP} = \Or(\gamma)(\cP)$ of Poisson brackets $\cP$ from $\Id$-cocycles $\gamma \in \ker \Id$.
% RB: But do we really solve a problem? Maybe all infinitesimal symmetries are zero, for a given Poisson structure.

Let $\gamma$ be an unoriented graph on $k$ vertices and let $p_1$,\ $\ldots$,\ $p_k \in \Tpoly(%\c
M)$ be a $k$-tuple of multivectors.
%%%
Not yet at the level of Lie algebras but at the level of two operads with the respective graph- and endomorphism insertions~$\vec{\circ}$, let the linear mapping~% 
%-- which we denote by 
$\littleOr$ %-- is 
be given by the formula~\cite{Ascona96}
\[ \littleOr(\gamma)(p_1,\ldots,p_k) (\boldsymbol{x},\boldsymbol{\xi})
\coloneqq \operatorname{mult}_k \bigg(\prod_{(i,j)\in E(\gamma)} \vec{\Delta}_{ij} (p_1 \otimes \ldots \otimes p_k) \bigg) (\boldsymbol{x},\boldsymbol{\xi}),
\]
where for each edge $(i,j) = e_{ij}$ in the graph $\gamma$, the operator
$\Delta_{ij}\colon e_{ij} \mapsto \smash{\bigl(i \xrightarrow{\ \ell\ } j\bigr) + \bigl(i \xleftarrow{\ \ell\ } j\bigr)}$,
%%%
\[ \vec{\Delta}_{ij} = \sum_{\ell=1}^r \bigg(\frac{\vec{\partial}}{\partial x^\ell_{(j)}} \frac{\vec{\partial}}{\partial \xi^{(i)}_\ell} + \frac{\vec{\partial}}{\partial \xi^{(j)}_\ell} \frac{\vec{\partial}}{\partial x^\ell_{(i)}} \bigg),
\]
acts on the $i^{\text{th}}$ and $j^{\text{th}}$ factors % elements, terms
in the ordered tensor product of arguments $p_1,\ldots,p_k$.
By construction, the right-to-left ordering of the operators $\vec{\Delta}_{ij}$ is inherited % induced
from the wedge ordering of edges $E(\gamma)$ in the graph $\gamma$: the operator corresponding to the firstmost edge acts first.\footnote{By construction, own grading of the endomorphism $\littleOr(\gamma)$ equals minus the number of edges in~$\gamma$ (because each edge differentiates one~$\xi_\ell$): $|\littleOr(\gamma)|=-|E(\gamma)|$, cf.\ Table~\ref{TabParallel}.}
%%%
The operator $\operatorname{mult}_k$, acting at the end of the day, is the ordered multiplication of the resulting terms in $\prod_{(i,j)} \vec{\Delta}_{ij}(p_1\otimes\ldots\otimes p_k)$.

%\begin{claim}[\cite{Jost2013,WillwacherGRT}]
It can be seen (\cite{Jost2013,WillwacherGRT}) that
the graph insertions $\vec{\circ}_i$ are mapped by $\littleOr$ to the insertions $\vec{\circ}_i$ of endomorphisms: %\footnote{%Let us support this by an example.
%\textit{Example.} 
%\begin{example}
%\[
% \text{TODO: big calculation.} 
%\]
%\end{example}
%} 
%%%
$\littleOr(\gamma_1 \mathbin{\vec{\circ}_i} \gamma_2) = \littleOr(\gamma_1) \mathbin{\vec{\circ}_i} \littleOr(\gamma_2)$.
%\end{claim}
%%%
Consequently, the sum of insertions $\vec{\circ}$ goes --\,under~$\littleOr$\,-- to the sum of insertions~$\vec{\circ}$.
% TODO: preserves the action of S_k
The mapping $\littleOr$ induces the linear mapping~$\Or$ taking %of 
graphs to the space of graded-skew endomorphisms $\End_{\text{skew}}(\Tpoly(%\c
M)[1])$. We reach the important equality:
  %%% No skew-symmetrization needed: the Schouten bracket 
  %%% is shifter-graded skew-symmetric at once.
  %%% => If skew-symmetrize, divide by factorial!
\begin{equation}\label{EqSchoutenPi}
\Or(\stick) = \pi_S, \qquad \text{i.e.}\quad \Or(\stick %\gamma
)(p_1,p_2) = (-)^{\bar{p}_1}\schouten{p_1,\ p_2} \quad \text{for } p_1,p_2 \in \Tpoly(%\c
M)[1]. 
\end{equation}% TODO: constant?
% Here \pi_S gets the normalization sign by definition
%%%
Recall also that both the domain and image of $\Or$, i.e.\ graphs with wedge ordering of edges and their skew-symmetrized images in the space $\End_{\text{skew}}(\Tpoly(%\c
M)[1])$ carry the respective Lie algebra structures.
%We conclude that
The conclusion is this:

\begin{prop}
The mapping $\Or\colon \bigoplus_k \Gra{k} \to \End_{\text{\textup{skew}}}^{*,*}(\Tpoly(%\c
M)[1])$ is a Lie algebra morphism\textup{:} $\Or([\gamma,\beta]) = [\Or(\gamma), \Or(\beta)]_{\NR}$.
\end{prop}

\begin{cor}
$\Or(\Id(\gamma)) = \Or([\stick,\gamma]) = [\Or(\stick), \Or(\gamma)]_{\NR} = [\pi_S, \Or(\gamma)]_{\NR}$.
\end{cor}

% TODO RB Remark: MC elements, Deformation complexes
Let there be $k$ vertices and $2k-2$ edges in~$\gamma$, whence $k+1$ vertices in~$\Id(\gamma)$.
   %But let us 
Evaluating both sides of the endomorphism equality $\Or(\Id(\gamma)) = [\pi_S, \Or(\gamma)]_{\NR}$ at a tuple of Poisson bi-vectors~$\cP$, %: specifically, .
   %So % by def
we have that $\Or([\stick, \gamma])(\cP\otimes\ldots\otimes \cP)={}$
\begin{align}
&= (\insertion{\pi_S}{\Or(\gamma)})(\cP\otimes\ldots\otimes \cP) - %{} \\
  %\hskip 1em -
(-)^{( |\pi_S| = -1 ) \cdot ( |\Or(\gamma)|=-|E(\gamma)| )}
(\insertion{\Or(\gamma)}{\pi_S})(\cP\otimes\ldots\otimes \cP) \notag \\
&= \Or(\gamma)(\pi_S(\cP,\cP),\underline{\cP,\ldots,\cP\ }_{k-1}) + \ldots + \Or(\gamma)(\underline{\cP,\ldots,\cP\ }_{k-1},\pi_S(\cP,\cP)) - \notag \\
& \hskip 4em -\pi_S(\Or(\gamma)(\underline{\cP,\ldots,\cP\ }_{k}), \cP) - \pi_S(\cP,\Or(\gamma)(\underline{\cP,\ldots,\cP\ }_{k})). \label{EqVal} %\nonumber
\end{align}

\begin{thm}
Whenever $\cP$ is a Poisson bi-vector so that $\pi_S(\cP,\cP) = 0 = \schouten{\cP,\cP}$, and whenever $\gamma \in \ker \Id$ is a cocycle on $k$ vertices and $2k-2$ edges \textup{(}so that $[\stick,\gamma] = 0$\textup{),} then $\Or(\gamma)(\underline{\cP\otimes\ldots\otimes \cP\ }_{k})$ is a Poisson cocycle \textup{(}so that $\schouten{\cP,\Or(\gamma)(\cP,\ldots,\cP)} \doteq 0$ modulo the Jacobi identity $\schouten{\cP,\cP} = 0$ for the Poisson structure\textup{)}.
\end{thm}

\begin{proof}
This is immediate from \eqref{EqVal}: its l.-h.s.\ %left-hand side 
vanishes by $\gamma \in \ker \Id$; in its right-hand side, the Jacobiator $\pi_S(\cP,\cP)$ is an argument of endomorphisms which are \emph{linear}, hence all the $k$ terms %values 
in the minuend vanish.
The subtrahend remains, it yields %is 
$2$ times the cocycle condition $\Or(\gamma)(\cP,\ldots,\cP) \in \ker\, \schouten{\cP,\cdot}$.
% TODO: mention \pi_S symmetric for bi-vectors.
\end{proof}

% TODO RB: Emphasize the following more, by adding some text in front. It is important.

\begin{cor}[A realization of $\Diamond$ by Leibniz graphs]
\label{CorDiamond}
The operator $\Diamond$ in the factorization problem 
\[
\partial_\cP(\Or(\gamma)(\cP,\ldots,\cP)) = \Diamond(\cP, \schouten{\cP,\cP}), \qquad \gamma \in \ker \Id,
\]
is the sum of Leibniz graphs obtained from $\gamma$ by inserting the Jacobiator $\schouten{\cP,\cP}$ into one of its vertices \textup{(}%consecutively / 
by the Leibniz rule\textup{)} and skew-symmetrizing w.r.t.\ %with respect to 
the sinks.
\end{cor}

\begin{proof}[Constructive proof]
Indeed, as \eqref{EqVal} yields (with $\pi_S(\cP,\cP) = (-)^{2-1}\schouten{\cP,\cP}$) 
equality~\eqref{EqOrCocycle},
\[ \schouten{\cP, \Or(\gamma)(\cP,\ldots,\cP)} = \tfrac{1}{2}\big\{\Or(\gamma)(\schouten{\cP,\cP},\cP,\ldots,\cP) + \ldots + \Or(\gamma)(\cP,\ldots,\cP,\schouten{\cP,\cP})\big\}, \label{EqDiamond}
\]
%So, 
the left-hand side of the cocycle condition factors, in particular, through the explicitly given set of Leibniz graphs with $\schouten{\cP,\cP}$ in one vertex in the right-hand side.
\end{proof}

\begin{cor}\label{CorCoboundary}
Suppose that $\delta = \Id(\gamma)$ is a trivial $\Id$-cocycle in the graph complex\textup{:} let there be $k$ vertices and $2k-1$ edges in $\gamma$.
Then, reading \eqref{EqVal} again, we have that for $\cP$ Poisson, % TODO: reading it again, but for a different number of edges
\begin{multline}\label{EqFieldX}
\Or(\delta)(\cP,\ldots,\cP) = 0 + \ldots + 0 - \pi_S(\underbrace{\Or(\gamma)(\cP,\ldots,\cP)}_{1\text{-vector}},\cP) - \pi_S(\cP,\underbrace{\Or(\gamma)(\cP,\ldots,\cP)}_{1\text{-vector}}) ={}\\
 -(-)^{1-1} \schouten{\Or(\gamma)(\cP,\ldots,\cP),\cP} - (-)^{2-1}\schouten{\cP,\Or(\gamma)(\cP,\ldots,\cP)} 
= 2\,\schouten{\cP, \Or(\gamma)(\cP,\ldots,\cP)}.
\end{multline}
%The equality $\Or(\delta)(\cP,\ldots,\cP) = 2\,\schouten{\cP, \Or(\gamma)(\cP,\ldots,\cP)}$
Equality~\eqref{EqFieldX}
provides the composition of the $1$-vector field $\cX(\cP) \coloneqq 2\, \Or(\gamma)(\cP,\ldots,\cP)$ trivializing $\Or(\delta)(\cP,\ldots,\cP) = \schouten{\cP, \cX(\cP)}$ in the Poisson cohomology.
\end{cor}

\begin{rem}\label{RemImproperInCoboundary}
From the above proof we also recognize the composition of Leibniz graphs (i.e.\ improper terms which vanish by the Jacobi identity $\schouten{\cP,\cP} = 0$) in the factorization problem
\[ \Or(\delta)(\cP,\ldots,\cP) - \schouten{\cP, \cX(\cP)} \doteq \nabla(\cP,\schouten{\cP,\cP}). \]
Namely, it is the terms $\Or(\gamma)(\pi_S(\cP,\cP),\cP,\ldots,\cP) + \ldots + \Or(\gamma)(\cP,\ldots,\cP,\pi_S(\cP,\cP))$ from~\eqref{EqVal}.
\end{rem}

\subsection{The morphism $\Or$ at work: examples}\label{SecExamples}
\noindent%
The following collection % set
of examples illustrates (\textit{i}) %$\bullet$ 
the construction of infinitesimal symmetries $\dot{\cP} = \cQ(\cP)$ for Poisson structures $\cP$ by orienting cocycles $\boldsymbol{\gamma} \in \ker \Id$, so that $\cQ = \Or(\boldsymbol{\gamma})$,
and (\textit{ii}) %$\bullet$~
the construction of trivializing vector fields $\cX = 2\Or(\boldsymbol{\gamma})$ in $\cQ = \Or(\Id(\boldsymbol{\gamma}))$.
At the same time, we detect (\textit{iii}) %$\bullet$~
the non-uniqueness of factorizations $\schouten{\cP,\cQ(\cP)} = \Diamond(\cP,\schouten{\cP,\cP})$ %(already) 
for such cocycles and flows.

We remember % keep in mind
that the (iterated commutators of the)
infinite sequence of $\Id$-cocycles $\boldsymbol{\gamma}_{2\ell+1}$, marked by $(2\ell+1)$-gon wheel graphs (see~\cite{WillwacherGRT}), is a regular source of universal symmetries for Poisson structures.
Moreover, no flows $\dot{\cP} = \cQ(\cP)$ other than these ones, $\cQ %_{2\ell+1}
(\cP) = \Or(\boldsymbol{\gamma} %_{2\ell+1}
)(\cP)$, are currently known (under the assumption that the cocycles~$\boldsymbol{\gamma}$ be sums of connected graphs).
% TODO RB: cohomology of stable polyvectorfields is graph complex, so nothing more?
% RB: See arXiv:1110.3762

Let us remark finally that it is also an \textbf{open problem} whether these flows, $\cQ %_{2\ell+1}
(\cP) = \Or(\boldsymbol{\gamma} %_{2\ell+1}
)(\cP)$, can be Poisson cohomology nontrivial, that is $\cQ \neq \schouten{\cP,\cX}$ for some Poisson structure $\cP$ and a globally defined vector field~$\cX$ on an affine manifold~$%\c
M$.

\begin{example}[The tetrahedron $\boldsymbol{\gamma}_3$]
For the tetrahedron $\boldsymbol{\gamma}_3 \in \ker \Id$, i.e.\ the full graph \raisebox{-4.75pt}[0pt][0pt]{
{\unitlength=0.75mm
\begin{picture}(7,7)(-3.5,-3.5)
\put(3.5,3.5){\circle*{1}}
\put(-3.5,3.5){\circle*{1}}
\put(-3.5,-3.5){\circle*{1}}
\put(3.5,-3.5){\circle*{1}}
\put(3.5,3.5){\line(-1,0){7}}
\put(-3.5,3.5){\line(0,-1){7}}
\put(-3.5,-3.5){\line(1,0){7}}
\put(3.5,-3.5){\line(0,1){7}}
\qbezier[30](3.5,3.5)(2,2)(0.75,0.75)
\qbezier[30](-3.5,-3.5)(-2,-2)(-1,-1)
%\put(3.5,3.5){\line(-1,-1){2.5}}
%\put(-3.5,-3.5){\line(1,1){2.5}}
\put(3.5,-3.5){\line(-1,1){7}}
\end{picture}
    }} on $4$ vertices and $6$ edges (see \cite{Ascona96}), both the Kontsevich flow $\dot{\cP} = \cQ_{1:6/2}(\cP)$ and the factorizing operator $\Diamond$ in the problem $\schouten{\cP, \cQ_{1:6/2}(\cP)} = \Diamond(\cP,\schouten{\cP,\cP})$ are presented in \cite{f16} (cf. \cite{Kontsevich2017Bourbaki}).
The operator $\Diamond$ is % up to a constant factor
of the form given by Corollary \ref{CorDiamond}.
\end{example}

\begin{example}[The pentagon-wheel cocycle $\boldsymbol{\gamma}_5 \in \ker \Id$]%
%\vskip 1.5em
{\unitlength=1mm
\begin{picture}(0,0)(0,7.5)
\put(0,0){$\boldsymbol{\gamma}_5 = {}$}
%\raisebox{0pt}[1pt][1pt]{
\put(20,0){
{\unitlength=0.3mm
\begin{picture}(55,53)(5,-5)%(27.5,23)
\put(27.5,8.5){\circle*{3}}
\put(0,29.5){\circle*{3}}
\put(-27.5,8.5){\circle*{3}}
\put(-17.5,-23.75){\circle*{3}}
\put(17.5,-23.75){\circle*{3}}
\qbezier%[40]
(27.5,8.5)(0,29.5)(0,29.5)
\qbezier%[40]
(0,29.5)(-27.5,8.5)(-27.5,8.5)
\qbezier%[40]
(-27.5,8.5)(-17.5,-23.75)(-17.5,-23.75)
\qbezier%[40]
(-17.5,-23.75)(17.5,-23.75)(17.5,-23.75)
\qbezier%[40]
(17.5,-23.75)(27.5,8.5)(27.5,8.5)
%%%
\put(0,0){\circle*{3}}
\qbezier%[30]
(27.5,8.5)(0,0)(0,0)
\qbezier%[30]
(0,29.5)(0,0)(0,0)
\qbezier%[30]
(-27.5,8.5)(0,0)(0,0)
\qbezier%[30]
(-17.5,-23.75)(0,0)(0,0)
\qbezier%[30]
(17.5,-23.75)(0,0)(0,0)
\end{picture}
}
}
%}
\put(30,0){${}+\dfrac{5}{2}$}
\put(50,0){
{\unitlength=0.4mm
%\raisebox{0pt}[1pt][1pt]{
\begin{picture}(50,30)(0,-4)%(27.5,23)
\put(12,0){\circle*{2.5}}
\put(-12,0){\circle*{2.5}}
\put(25,15){\circle*{2.5}}
\put(-25,15){\circle*{2.5}}
\put(-25,-15){\circle*{2.5}}
\put(25,-15){\circle*{2.5}}
%%%
\put(-12,0){\line(1,0){24}}
\put(-25,15){\line(1,0){50}}
\put(-25,-15){\line(1,0){50}}
\put(-25,-15){\line(0,1){32}}
\put(25,15){\line(0,-1){32}}
%%%
\qbezier%[30]
(25,15)(12,0)(12,0)
\qbezier%[30]
(-25,15)(-12,0)(-12,0)
\qbezier%[30]
(-25,-15)(-12,0)(-12,0)
\qbezier%[30]
(25,-15)(12,0)(12,0)
%%%
\put(-12.5,17){\oval(25,10)[t]}
\put(12.5,-17){\oval(25,10)[b]}
\put(0,2){\line(0,1){11}}
\put(0,-2){\line(0,-1){11}}
\end{picture}
%}
}%
}
%\put(50,20){.}
\end{picture}%
}%
%\vskip 2.5em

\smallskip
\noindent\parbox{0.60\textwidth}{%
For the pentagon\/-\/wheel cocycle,
the set of oriented Kon\-tse\-vich graphs that encode the flow $\dot{\cP} = \Or(\boldsymbol{\gamma}_5)(\cP)$ is listed~in \cite{sqs17}.\hfill
The resulting differential polynomial expression of this infinitesimal symmetry% 
}
%%%
\\[1pt]
is available % given
in Appendix \ref{AppDiffPoly5} below.
%%% either there [in Appendix] or here if JPCS==10pp.
%%%
But the factorizing operator for $\Or(\boldsymbol{\gamma}_5)(\cP)$ reported in \cite{sqs17}, i.e.\ expressing $\schouten{\cP, \Or(\boldsymbol{\gamma}_5)(\cP)}$ as a sum of Leibniz graphs, is \emph{different} from the operator $\Diamond$ which Corollary \ref{CorDiamond} provides for the cocycle $\boldsymbol{\gamma}_5$.
This demonstrates that such operators can be non-unique (as one obtains it in this particular example).\footnote{We say that two Leibniz graphs (i.e.\ graphs with a tri\/-\/vector $\schouten{\cP,\cP}$ in a vertex) are \emph{adjacent vertices} in the Leibniz meta\/-\/graph if the expansions of these Leibniz graphs have at least one Kontsevich oriented graph in common.
(In the meta\/-\/graphs, multiple edges are allowed.)
The known existence of several factorizations, $\schouten{\cP,\cQ}=\Diamond_1\bigl(\cP$,\ $\schouten{\cP,\cP}\bigr) = \Diamond_2\bigl(\cP$,\ $\schouten{\cP,\cP}\bigr)$,
into Leibniz graphs reveals the identities $(\Diamond_1-\Diamond_2)\,\bigl(\cP$,\ $\schouten{\cP,\cP}\bigr)\equiv0$ for $\Diamond_1\neq\Diamond_2$, that is, a nontrivial topology of the meta\/-\/graph. Its study is an \textbf{open problem}.}
\end{example}

\begin{example}[Coboundary $\boldsymbol{\delta}_6 = \Id(\boldsymbol{\beta}_6)$]\label{ExBeta6}
Take the only nonzero (with% 
%\vskip .1em 
{\unitlength=1mm
\begin{picture}(0,0)(-6,1.5)
\put(0,0){$\boldsymbol{\beta}_6 ={}$} %\hskip 1.8em 
%\text{
%\raisebox{3pt}[1pt][1pt]{
\put(16,-4){
{\unitlength=0.09mm
\begin{picture}(0,0)%(0,60)
\put(0,0){\circle*{10}}
\put(0,0){\line(0,1){120}}
\put(0,120){\circle*{10}}
\put(0,120){\line(1,0){120}}
\put(120,120){\circle*{10}}
\put(0,0){\line(1,0){120}}
\put(120,0){\circle*{10}}
\put(120,0){\line(0,1){120}}
\put(60,60){\circle*{10}}
\qbezier[60](60,60)(90,90)(120,120)
\qbezier[60](60,60)(30,90)(0,120)
\qbezier[60](60,60)(90,30)(120,0)
\qbezier[60](60,60)(30,30)(0,0)
\put(-60,60){\circle*{10}}
\qbezier[60](-60,60)(-30,90)(0,120)
\qbezier[60](-60,60)(-30,30)(0,0)
\put(-60,60){\line(1,0){120}}
%\put(-25,-10){\tiny$5$}
%\put(-25,115){\tiny$1$}
%\put(131,115){\tiny$4$}
%\put(131,-10){\tiny$2$}
%\put(73,53){\tiny$6$}
%\put(-85,53){\tiny$3$}
\end{picture}
}}
\end{picture}
%}
%}
}
%$$ 
%\vskip .8em \noindent 
%$$

\smallskip
\noindent\parbox{0.74\textwidth}{%
respect to the wedge ordering of edges)
connected graph~$\boldsymbol{\beta}_6$
on~six%$6$%
}
\\[1pt]
vertices and $11$ edges,
and put $\boldsymbol{\delta}_6 = \smash{\Id(\boldsymbol{\beta}_6)} \in \ker \Id$ (indeed,
$\Id^2 = 0$).
In view of Corollary~\ref{CorCoboundary} and Remark \ref{RemImproperInCoboundary}, we verify the decomposition, 
\[ \Or(\Id(\boldsymbol{\beta}_6))(\cP) = \schouten{\cP, \cX(\cP)} + \nabla(\cP,\schouten{\cP,\cP}),
\] into the Poisson cohomology trivial and improper terms.
Indeed, the vector field~$\cX$ stems from $\Or(\boldsymbol{\beta}_6)(\cP,\ldots,\cP)$ and the improper part comes from the terms like $\Or(\boldsymbol{\beta}_6)(\\\schouten{\cP,\cP}, \ldots,\cP)$.
Interestingly, all the graphs from the $\partial_\cP$-\/exact term $\schouten{\cP,\cX}$ also appear in the improper terms, and in fact they cancel.
(There are $598$ graphs in the former and $2098$ in the latter; $2098 - 598 = 1500$, cf. Table \ref{TabNumGraphs} below.)
\end{example}

\begin{example}[The heptagon-wheel cocycle $\boldsymbol{\gamma}_7 \in \ker \Id$]
The $\Id$-cocycle starting with the heptagon\/-\/wheel graph is presented in~\cite{JNMP17}.
The flow $\dot{\cP} = \Or(\boldsymbol{\gamma}_7)(\cP)$ is realized by 37,185 Kontsevich graphs on $2$ sinks; they are listed in a standard format (see \cite[Implementation 1]{f16}) at \texttt{http://rburing.nl/gamma7.zip}. % TODO.
The factorizing operator $\Diamond$ is provided by Corollary \ref{CorDiamond} so that the validity of cocycle equation $\schouten{\cP,\Or(\boldsymbol{\gamma}_7)(\cP)} = \Diamond(\cP,\schouten{\cP,\cP})$ is verified experimentally.
(It would be unfeasible to solve this equation w.r.t.\ %with respect to 
the unknown coefficients of the Leibniz graphs in the right-hand side, pretending that a solution is not known from~\S\ref{SecOr}.
Note however that no uniqueness is claimed for this~$\Diamond$.)
\end{example}

\begin{table}[h!]
\caption{The number of graphs in the problem $\schouten{\cP, \Or(\gamma)(\cP)} = \Diamond(\cP, \schouten{\cP,\cP})$.}\label{TabNumGraphs}
\begin{center}
\begin{tabular}{l r r r r}
Cocycle: & $\boldsymbol{\gamma}_3$ & $\boldsymbol{\gamma}_5$ & $\boldsymbol{\delta}_6 = \Id(\boldsymbol{\beta}_6)$ & $\boldsymbol{\gamma}_7$ 
   % & $[\boldsymbol{\gamma}_3,\boldsymbol{\gamma}_5]$ ???
\\
\hline
\#vertices: & 4 & 6 & 7 & 8 \\
\#edges: & 6 & 10 & 12 & 14 \\
\#graphs: & 1 & 2 & 4 & 46 \\
\#or.graphs in $\cQ(\cP) = \Or(\gamma)(\cP,\ldots,\cP)$: & 3 & 167 & 1,500 & 37,185 \\
% TODO RB: number of skew graphs
\#or.graphs in $\schouten{\cP,\cQ(\cP)}$: & 39 & 3,495 & 35,949 & 1,003,611 \\
% TODO RB: number of skew graphs
\#skew Leibniz graphs in $\Diamond(\cP,\schouten{\cP,\cP})$: & 8 & 843 & 9,556 & 293,654
\end{tabular}
\end{center}
\end{table}

\noindent\textbf{Implementation.}\quad
\noindent All calculations above were performed by using the software packages
\texttt{graph\symbol{"5F}complex-cpp} 
and \texttt{kontsevich\symbol{"5F}graph\symbol{"5F}series-cpp}, %~\cite{kgs-cpp}
which are released under the MIT free software license and available from
\texttt{https://github.com/rburing}.
Specifically, the programs \texttt{expanding\symbol{"5F}differential} and \texttt{kernel} have been used to find non-oriented graph cocycles $\gamma$, 
\texttt{orient} yields the sums of Kontsevich oriented graphs $\Or(\gamma)(\cP,\ldots,\cP)$ and sums of Leibniz graphs $\Or(\schouten{\cP,\cP},\ldots,\cP)$, and \texttt{schouten\symbol{"5F}bracket} implements the Schouten bracket.
The program \texttt{leibniz\symbol{"5F}expand} expands sums of Leibniz graphs into Kontsevich graphs, and \texttt{reduce\symbol{"5F}mod\symbol{"5F}skew} reduces sums of Kontsevich oriented graphs modulo skew-symmetry, $L\prec R = - R\prec L$, of the Left $\prec$ Right mark\/-\/up of outgoing edges.
% TODO: trick with normal forms & sorting

{\small
\subsubsection*{Acknowledgements}
The research of RB was supported in part by IM %Institut f\"ur Mathematik
JGU Mainz; AVK was partially supported by JBI RUG project 106552.
The authors are grateful to the organizers of the 32nd International Colloquium on Group Theoretical Methods in Physics (\textsc{Group32}, 9 -- 13 July 2018, CVUT Prague, Czech Republic) for partial financial support (for RB) and warm hospitality.
The authors thank M.~Kontsevich, A.~Sharapov, and T.~Willwacher % + From Bedlewo?
for helpful discussions.
A part of this research was done while RB was visiting at RUG and AVK was visiting at JGU Mainz (supported by IM via project~5020). % RB project

}%\par=\empty line: to make paragraph with \small line intervals

\newpage

\appendix

\section{The differential polynomial flow $\dot{\cP}=\Or(\boldsymbol{\gamma}_5)(\cP)$}
\label{AppDiffPoly5}
\noindent%
Here is the value $\cQ_5(\cP)(f,g)$ of the bi\/-\/vector~$\cQ_5(\cP) = \Or(\gamma_5)(\cP)$ at two functions~$f,g$:\footnote{In every term, the Einstein summation convention works for each repeated index (i.e.\ once upper and another time lower), the indices running from~$1$ to the dimension~$\dim M <\infty$ of the affine Poisson manifold~$%\c
M$ at hand.}
   %first-order differential operator on two functions 
% RB: after some programming, we could also have it as a bi-vector field (shorter)

{\tiny%\footnotesize
\begin{align*}
10 \partial_{t} \partial_{m} \partial_{k} \cP^{ij} \partial_{p} \cP^{k\ell} \partial_{v} \partial_{r} \partial_{\ell} \cP^{mn} \partial_{n} \cP^{pq} \partial_{q} \cP^{rs} \partial_{s} \cP^{tv} \partial_{i} f \partial_{j} g 
-10 \partial_{p} \partial_{m} \partial_{k} \cP^{ij} \partial_{t} \cP^{k\ell} \partial_{v} \partial_{r} \partial_{\ell} \cP^{mn} \partial_{n} \cP^{pq} \partial_{q} \cP^{rs} \partial_{s} \cP^{tv} \partial_{i} f \partial_{j} g \\
+10 \partial_{r} \partial_{m} \cP^{ij} \partial_{t} \partial_{j} \cP^{k\ell} \partial_{v} \partial_{s} \partial_{\ell} \cP^{mn} \partial_{n} \cP^{pq} \partial_{p} \cP^{rs} \partial_{q} \cP^{tv} \partial_{i} f \partial_{k} g
-10 \partial_{r} \partial_{n} \cP^{ij} \partial_{t} \partial_{s} \partial_{j} \cP^{k\ell} \partial_{v} \partial_{k} \cP^{mn} \partial_{\ell} \cP^{pq} \partial_{p} \cP^{rs} \partial_{q} \cP^{tv} \partial_{i} f \partial_{m} g \\
+10 \partial_{p} \partial_{m} \cP^{ij} \partial_{t} \partial_{j} \cP^{k\ell} \partial_{v} \partial_{r} \partial_{\ell} \cP^{mn} \partial_{n} \cP^{pq} \partial_{q} \cP^{rs} \partial_{s} \cP^{tv} \partial_{i} f \partial_{k} g
-10 \partial_{t} \partial_{n} \cP^{ij} \partial_{v} \partial_{r} \partial_{j} \cP^{k\ell} \partial_{p} \partial_{k} \cP^{mn} \partial_{\ell} \cP^{pq} \partial_{q} \cP^{rs} \partial_{s} \cP^{tv} \partial_{i} f \partial_{m} g \\
-10 \partial_{t} \partial_{m} \cP^{ij} \partial_{p} \partial_{j} \cP^{k\ell} \partial_{v} \partial_{r} \partial_{\ell} \cP^{mn} \partial_{n} \cP^{pq} \partial_{q} \cP^{rs} \partial_{s} \cP^{tv} \partial_{i} f \partial_{k} g
-10 \partial_{p} \partial_{n} \cP^{ij} \partial_{t} \partial_{r} \partial_{j} \cP^{k\ell} \partial_{v} \partial_{k} \cP^{mn} \partial_{\ell} \cP^{pq} \partial_{q} \cP^{rs} \partial_{s} \cP^{tv} \partial_{i} f \partial_{m} g \\
-10 \partial_{t} \partial_{p} \cP^{ij} \partial_{q} \partial_{j} \cP^{k\ell} \partial_{\ell} \cP^{mn} \partial_{v} \partial_{r} \partial_{m} \cP^{pq} \partial_{n} \cP^{rs} \partial_{s} \cP^{tv} \partial_{i} f \partial_{k} g
+10 \partial_{s} \partial_{m} \cP^{ij} \partial_{j} \cP^{k\ell} \partial_{t} \partial_{p} \partial_{k} \cP^{mn} \partial_{\ell} \cP^{pq} \partial_{v} \partial_{n} \cP^{rs} \partial_{q} \cP^{tv} \partial_{i} f \partial_{r} g \\
+10 \partial_{t} \partial_{p} \cP^{ij} \partial_{j} \cP^{k\ell} \partial_{q} \partial_{k} \cP^{mn} \partial_{v} \partial_{r} \partial_{\ell} \cP^{pq} \partial_{n} \cP^{rs} \partial_{s} \cP^{tv} \partial_{i} f \partial_{m} g
-10 \partial_{t} \partial_{m} \cP^{ij} \partial_{j} \cP^{k\ell} \partial_{v} \partial_{p} \partial_{k} \cP^{mn} \partial_{\ell} \cP^{pq} \partial_{q} \partial_{n} \cP^{rs} \partial_{s} \cP^{tv} \partial_{i} f \partial_{r} g \\
-10 \partial_{r} \partial_{m} \cP^{ij} \partial_{j} \cP^{k\ell} \partial_{v} \partial_{p} \partial_{k} \cP^{mn} \partial_{\ell} \cP^{pq} \partial_{n} \cP^{rs} \partial_{s} \partial_{q} \cP^{tv} \partial_{i} f \partial_{t} g
+10 \partial_{t} \partial_{p} \cP^{ij} \partial_{v} \partial_{r} \partial_{j} \cP^{k\ell} \partial_{q} \partial_{k} \cP^{mn} \partial_{\ell} \cP^{pq} \partial_{n} \cP^{rs} \partial_{s} \cP^{tv} \partial_{i} f \partial_{m} g \\
-10 \partial_{t} \partial_{m} \cP^{ij} \partial_{v} \partial_{s} \partial_{j} \cP^{k\ell} \partial_{k} \cP^{mn} \partial_{\ell} \cP^{pq} \partial_{p} \partial_{n} \cP^{rs} \partial_{q} \cP^{tv} \partial_{i} f \partial_{r} g
+10 \partial_{p} \partial_{m} \cP^{ij} \partial_{t} \partial_{s} \partial_{j} \cP^{k\ell} \partial_{k} \cP^{mn} \partial_{\ell} \cP^{pq} \partial_{v} \partial_{n} \cP^{rs} \partial_{q} \cP^{tv} \partial_{i} f \partial_{r} g \\
-10 \partial_{t} \partial_{r} \cP^{ij} \partial_{v} \partial_{s} \partial_{j} \cP^{k\ell} \partial_{p} \partial_{k} \cP^{mn} \partial_{\ell} \cP^{pq} \partial_{n} \cP^{rs} \partial_{q} \cP^{tv} \partial_{i} f \partial_{m} g
+10 \partial_{r} \partial_{p} \cP^{ij} \partial_{j} \cP^{k\ell} \partial_{t} \partial_{k} \cP^{mn} \partial_{v} \partial_{n} \partial_{\ell} \cP^{pq} \partial_{q} \cP^{rs} \partial_{s} \cP^{tv} \partial_{i} f \partial_{m} g \\
+10 \partial_{r} \partial_{p} \cP^{ij} \partial_{t} \partial_{s} \partial_{j} \cP^{k\ell} \partial_{v} \partial_{k} \cP^{mn} \partial_{\ell} \cP^{pq} \partial_{n} \cP^{rs} \partial_{q} \cP^{tv} \partial_{i} f \partial_{m} g
+10 \partial_{t} \partial_{p} \cP^{ij} \partial_{j} \cP^{k\ell} \partial_{r} \partial_{k} \cP^{mn} \partial_{v} \partial_{n} \partial_{\ell} \cP^{pq} \partial_{q} \cP^{rs} \partial_{s} \cP^{tv} \partial_{i} f \partial_{m} g \\
-10 \partial_{t} \partial_{n} \cP^{ij} \partial_{j} \cP^{k\ell} \partial_{v} \partial_{r} \partial_{p} \partial_{k} \cP^{mn} \partial_{\ell} \cP^{pq} \partial_{q} \cP^{rs} \partial_{s} \cP^{tv} \partial_{i} f \partial_{m} g
-10 \partial_{t} \partial_{r} \partial_{p} \partial_{m} \cP^{ij} \partial_{v} \partial_{j} \cP^{k\ell} \partial_{\ell} \cP^{mn} \partial_{n} \cP^{pq} \partial_{q} \cP^{rs} \partial_{s} \cP^{tv} \partial_{i} f \partial_{k} g \\
-10 \partial_{t} \partial_{m} \cP^{ij} \partial_{v} \partial_{r} \partial_{p} \partial_{j} \cP^{k\ell} \partial_{\ell} \cP^{mn} \partial_{n} \cP^{pq} \partial_{q} \cP^{rs} \partial_{s} \cP^{tv} \partial_{i} f \partial_{k} g
-10 \partial_{t} \partial_{r} \partial_{p} \partial_{n} \cP^{ij} \partial_{j} \cP^{k\ell} \partial_{v} \partial_{k} \cP^{mn} \partial_{\ell} \cP^{pq} \partial_{q} \cP^{rs} \partial_{s} \cP^{tv} \partial_{i} f \partial_{m} g \\
-10 \partial_{t} \partial_{p} \cP^{ij} \partial_{v} \partial_{r} \partial_{q} \partial_{j} \cP^{k\ell} \partial_{\ell} \cP^{mn} \partial_{m} \cP^{pq} \partial_{n} \cP^{rs} \partial_{s} \cP^{tv} \partial_{i} f \partial_{k} g
+10 \partial_{t} \partial_{s} \partial_{p} \partial_{m} \cP^{ij} \partial_{j} \cP^{k\ell} \partial_{k} \cP^{mn} \partial_{\ell} \cP^{pq} \partial_{v} \partial_{n} \cP^{rs} \partial_{q} \cP^{tv} \partial_{i} f \partial_{r} g \\
+2 \partial_{t} \partial_{r} \partial_{p} \partial_{m} \partial_{k} \cP^{ij} \partial_{v} \cP^{k\ell} \partial_{\ell} \cP^{mn} \partial_{n} \cP^{pq} \partial_{q} \cP^{rs} \partial_{s} \cP^{tv} \partial_{i} f \partial_{j} g
-5 \partial_{p} \partial_{m} \partial_{k} \cP^{ij} \partial_{t} \partial_{r} \cP^{k\ell} \partial_{v} \partial_{\ell} \cP^{mn} \partial_{n} \cP^{pq} \partial_{q} \cP^{rs} \partial_{s} \cP^{tv} \partial_{i} f \partial_{j} g \\
+5 \partial_{t} \partial_{m} \partial_{k} \cP^{ij} \partial_{r} \partial_{p} \cP^{k\ell} \partial_{v} \partial_{\ell} \cP^{mn} \partial_{n} \cP^{pq} \partial_{q} \cP^{rs} \partial_{s} \cP^{tv} \partial_{i} f \partial_{j} g
-5 \partial_{p} \partial_{m} \partial_{k} \cP^{ij} \partial_{r} \cP^{k\ell} \partial_{t} \partial_{\ell} \cP^{mn} \partial_{v} \partial_{n} \cP^{pq} \partial_{q} \cP^{rs} \partial_{s} \cP^{tv} \partial_{i} f \partial_{j} g \\
-5 \partial_{p} \partial_{m} \partial_{k} \cP^{ij} \partial_{t} \cP^{k\ell} \partial_{r} \partial_{\ell} \cP^{mn} \partial_{v} \partial_{n} \cP^{pq} \partial_{q} \cP^{rs} \partial_{s} \cP^{tv} \partial_{i} f \partial_{j} g
+5 \partial_{t} \partial_{p} \cP^{ij} \partial_{v} \partial_{j} \cP^{k\ell} \partial_{\ell} \cP^{mn} \partial_{r} \partial_{m} \cP^{pq} \partial_{n} \cP^{rs} \partial_{s} \partial_{q} \cP^{tv} \partial_{i} f \partial_{k} g \\
+5 \partial_{v} \partial_{r} \cP^{ij} \partial_{j} \cP^{k\ell} \partial_{k} \cP^{mn} \partial_{m} \partial_{\ell} \cP^{pq} \partial_{p} \partial_{n} \cP^{rs} \partial_{s} \partial_{q} \cP^{tv} \partial_{i} f \partial_{t} g %\\
+5 \partial_{r} \partial_{m} \cP^{ij} \partial_{t} \partial_{j} \cP^{k\ell} \partial_{s} \partial_{\ell} \cP^{mn} \partial_{n} \cP^{pq} \partial_{v} \partial_{p} \cP^{rs} \partial_{q} \cP^{tv} \partial_{i} f \partial_{k} g \\
-5 \partial_{r} \partial_{n} \cP^{ij} \partial_{t} \partial_{j} \cP^{k\ell} \partial_{v} \partial_{k} \cP^{mn} \partial_{\ell} \cP^{pq} \partial_{p} \cP^{rs} \partial_{s} \partial_{q} \cP^{tv} \partial_{i} f \partial_{m} g
+5 \partial_{p} \partial_{m} \cP^{ij} \partial_{r} \partial_{j} \cP^{k\ell} \partial_{t} \partial_{\ell} \cP^{mn} \partial_{v} \partial_{n} \cP^{pq} \partial_{q} \cP^{rs} \partial_{s} \cP^{tv} \partial_{i} f \partial_{k} g \\
-5 \partial_{r} \partial_{n} \cP^{ij} \partial_{t} \partial_{j} \cP^{k\ell} \partial_{p} \partial_{k} \cP^{mn} \partial_{v} \partial_{\ell} \cP^{pq} \partial_{q} \cP^{rs} \partial_{s} \cP^{tv} \partial_{i} f \partial_{m} g
+5 \partial_{p} \partial_{m} \cP^{ij} \partial_{t} \partial_{j} \cP^{k\ell} \partial_{r} \partial_{\ell} \cP^{mn} \partial_{v} \partial_{n} \cP^{pq} \partial_{q} \cP^{rs} \partial_{s} \cP^{tv} \partial_{i} f \partial_{k} g \\
-5 \partial_{t} \partial_{n} \cP^{ij} \partial_{r} \partial_{j} \cP^{k\ell} \partial_{p} \partial_{k} \cP^{mn} \partial_{v} \partial_{\ell} \cP^{pq} \partial_{q} \cP^{rs} \partial_{s} \cP^{tv} \partial_{i} f \partial_{m} g
+5 \partial_{r} \partial_{n} \cP^{ij} \partial_{s} \partial_{j} \cP^{k\ell} \partial_{t} \partial_{k} \cP^{mn} \partial_{\ell} \cP^{pq} \partial_{v} \partial_{p} \cP^{rs} \partial_{q} \cP^{tv} \partial_{i} f \partial_{m} g \\
+5 \partial_{r} \partial_{m} \cP^{ij} \partial_{t} \partial_{j} \cP^{k\ell} \partial_{v} \partial_{\ell} \cP^{mn} \partial_{n} \cP^{pq} \partial_{p} \cP^{rs} \partial_{s} \partial_{q} \cP^{tv} \partial_{i} f \partial_{k} g
+5 \partial_{p} \partial_{n} \cP^{ij} \partial_{t} \partial_{j} \cP^{k\ell} \partial_{r} \partial_{k} \cP^{mn} \partial_{v} \partial_{\ell} \cP^{pq} \partial_{q} \cP^{rs} \partial_{s} \cP^{tv} \partial_{i} f \partial_{m} g \\
-5 \partial_{r} \partial_{m} \cP^{ij} \partial_{p} \partial_{j} \cP^{k\ell} \partial_{t} \partial_{\ell} \cP^{mn} \partial_{v} \partial_{n} \cP^{pq} \partial_{q} \cP^{rs} \partial_{s} \cP^{tv} \partial_{i} f \partial_{k} g
+5 \partial_{p} \partial_{n} \cP^{ij} \partial_{r} \partial_{j} \cP^{k\ell} \partial_{t} \partial_{k} \cP^{mn} \partial_{v} \partial_{\ell} \cP^{pq} \partial_{q} \cP^{rs} \partial_{s} \cP^{tv} \partial_{i} f \partial_{m} g \\
-5 \partial_{t} \partial_{m} \cP^{ij} \partial_{p} \partial_{j} \cP^{k\ell} \partial_{r} \partial_{\ell} \cP^{mn} \partial_{v} \partial_{n} \cP^{pq} \partial_{q} \cP^{rs} \partial_{s} \cP^{tv} \partial_{i} f \partial_{k} g
+5 \partial_{s} \partial_{m} \cP^{ij} \partial_{t} \partial_{j} \cP^{k\ell} \partial_{v} \partial_{k} \cP^{mn} \partial_{\ell} \cP^{pq} \partial_{p} \partial_{n} \cP^{rs} \partial_{q} \cP^{tv} \partial_{i} f \partial_{r} g \\
+5 \partial_{r} \partial_{p} \cP^{ij} \partial_{q} \partial_{j} \cP^{k\ell} \partial_{t} \partial_{\ell} \cP^{mn} \partial_{v} \partial_{m} \cP^{pq} \partial_{n} \cP^{rs} \partial_{s} \cP^{tv} \partial_{i} f \partial_{k} g
+5 \partial_{s} \partial_{m} \cP^{ij} \partial_{t} \partial_{j} \cP^{k\ell} \partial_{p} \partial_{k} \cP^{mn} \partial_{\ell} \cP^{pq} \partial_{v} \partial_{n} \cP^{rs} \partial_{q} \cP^{tv} \partial_{i} f \partial_{r} g \\
-5 \partial_{t} \partial_{p} \cP^{ij} \partial_{q} \partial_{j} \cP^{k\ell} \partial_{v} \partial_{\ell} \cP^{mn} \partial_{r} \partial_{m} \cP^{pq} \partial_{n} \cP^{rs} \partial_{s} \cP^{tv} \partial_{i} f \partial_{k} g
-5 \partial_{r} \partial_{n} \cP^{ij} \partial_{j} \cP^{k\ell} \partial_{t} \partial_{s} \partial_{k} \cP^{mn} \partial_{\ell} \cP^{pq} \partial_{v} \partial_{p} \cP^{rs} \partial_{q} \cP^{tv} \partial_{i} f \partial_{m} g \\
-5 \partial_{t} \partial_{r} \partial_{m} \cP^{ij} \partial_{s} \partial_{j} \cP^{k\ell} \partial_{\ell} \cP^{mn} \partial_{n} \cP^{pq} \partial_{v} \partial_{p} \cP^{rs} \partial_{q} \cP^{tv} \partial_{i} f \partial_{k} g
-5 \partial_{t} \partial_{p} \cP^{ij} \partial_{v} \partial_{q} \partial_{j} \cP^{k\ell} \partial_{\ell} \cP^{mn} \partial_{r} \partial_{m} \cP^{pq} \partial_{n} \cP^{rs} \partial_{s} \cP^{tv} \partial_{i} f \partial_{k} g \\
+5 \partial_{t} \partial_{s} \partial_{m} \cP^{ij} \partial_{j} \cP^{k\ell} \partial_{p} \partial_{k} \cP^{mn} \partial_{\ell} \cP^{pq} \partial_{v} \partial_{n} \cP^{rs} \partial_{q} \cP^{tv} \partial_{i} f \partial_{r} g
+5 \partial_{r} \partial_{m} \cP^{ij} \partial_{t} \partial_{s} \partial_{j} \cP^{k\ell} \partial_{v} \partial_{\ell} \cP^{mn} \partial_{n} \cP^{pq} \partial_{p} \cP^{rs} \partial_{q} \cP^{tv} \partial_{i} f \partial_{k} g \\
-5 \partial_{t} \partial_{r} \partial_{n} \cP^{ij} \partial_{s} \partial_{j} \cP^{k\ell} \partial_{v} \partial_{k} \cP^{mn} \partial_{\ell} \cP^{pq} \partial_{p} \cP^{rs} \partial_{q} \cP^{tv} \partial_{i} f \partial_{m} g
-5 \partial_{t} \partial_{m} \cP^{ij} \partial_{v} \partial_{p} \partial_{j} \cP^{k\ell} \partial_{r} \partial_{\ell} \cP^{mn} \partial_{n} \cP^{pq} \partial_{q} \cP^{rs} \partial_{s} \cP^{tv} \partial_{i} f \partial_{k} g \\
-5 \partial_{t} \partial_{p} \partial_{n} \cP^{ij} \partial_{r} \partial_{j} \cP^{k\ell} \partial_{v} \partial_{k} \cP^{mn} \partial_{\ell} \cP^{pq} \partial_{q} \cP^{rs} \partial_{s} \cP^{tv} \partial_{i} f \partial_{m} g
-5 \partial_{r} \partial_{m} \cP^{ij} \partial_{t} \partial_{s} \partial_{j} \cP^{k\ell} \partial_{\ell} \cP^{mn} \partial_{n} \cP^{pq} \partial_{v} \partial_{p} \cP^{rs} \partial_{q} \cP^{tv} \partial_{i} f \partial_{k} g \\
-5 \partial_{t} \partial_{r} \partial_{n} \cP^{ij} \partial_{j} \cP^{k\ell} \partial_{s} \partial_{k} \cP^{mn} \partial_{\ell} \cP^{pq} \partial_{v} \partial_{p} \cP^{rs} \partial_{q} \cP^{tv} \partial_{i} f \partial_{m} g
-5 \partial_{r} \partial_{n} \cP^{ij} \partial_{t} \partial_{j} \cP^{k\ell} \partial_{v} \partial_{s} \partial_{k} \cP^{mn} \partial_{\ell} \cP^{pq} \partial_{p} \cP^{rs} \partial_{q} \cP^{tv} \partial_{i} f \partial_{m} g \\
+5 \partial_{t} \partial_{r} \partial_{m} \cP^{ij} \partial_{s} \partial_{j} \cP^{k\ell} \partial_{v} \partial_{\ell} \cP^{mn} \partial_{n} \cP^{pq} \partial_{p} \cP^{rs} \partial_{q} \cP^{tv} \partial_{i} f \partial_{k} g
-5 \partial_{t} \partial_{n} \cP^{ij} \partial_{r} \partial_{j} \cP^{k\ell} \partial_{v} \partial_{p} \partial_{k} \cP^{mn} \partial_{\ell} \cP^{pq} \partial_{q} \cP^{rs} \partial_{s} \cP^{tv} \partial_{i} f \partial_{m} g \\
-5 \partial_{t} \partial_{p} \partial_{m} \cP^{ij} \partial_{v} \partial_{j} \cP^{k\ell} \partial_{r} \partial_{\ell} \cP^{mn} \partial_{n} \cP^{pq} \partial_{q} \cP^{rs} \partial_{s} \cP^{tv} \partial_{i} f \partial_{k} g
-5 \partial_{r} \partial_{m} \partial_{k} \cP^{ij} \partial_{t} \cP^{k\ell} \partial_{v} \partial_{\ell} \cP^{mn} \partial_{n} \cP^{pq} \partial_{p} \cP^{rs} \partial_{s} \partial_{q} \cP^{tv} \partial_{i} f \partial_{j} g \\
+5 \partial_{r} \partial_{m} \partial_{k} \cP^{ij} \partial_{p} \cP^{k\ell} \partial_{t} \partial_{\ell} \cP^{mn} \partial_{v} \partial_{n} \cP^{pq} \partial_{q} \cP^{rs} \partial_{s} \cP^{tv} \partial_{i} f \partial_{j} g
+5 \partial_{t} \partial_{m} \partial_{k} \cP^{ij} \partial_{p} \cP^{k\ell} \partial_{r} \partial_{\ell} \cP^{mn} \partial_{v} \partial_{n} \cP^{pq} \partial_{q} \cP^{rs} \partial_{s} \cP^{tv} \partial_{i} f \partial_{j} g \\
+5 \partial_{r} \partial_{m} \partial_{k} \cP^{ij} \partial_{t} \partial_{p} \cP^{k\ell} \partial_{\ell} \cP^{mn} \partial_{v} \partial_{n} \cP^{pq} \partial_{q} \cP^{rs} \partial_{s} \cP^{tv} \partial_{i} f \partial_{j} g
+5 \partial_{t} \partial_{m} \partial_{k} \cP^{ij} \partial_{r} \partial_{p} \cP^{k\ell} \partial_{\ell} \cP^{mn} \partial_{v} \partial_{n} \cP^{pq} \partial_{q} \cP^{rs} \partial_{s} \cP^{tv} \partial_{i} f \partial_{j} g \\
-5 \partial_{r} \partial_{n} \cP^{ij} \partial_{j} \cP^{k\ell} \partial_{t} \partial_{p} \partial_{k} \cP^{mn} \partial_{v} \partial_{\ell} \cP^{pq} \partial_{q} \cP^{rs} \partial_{s} \cP^{tv} \partial_{i} f \partial_{m} g 
+5 \partial_{t} \partial_{p} \partial_{m} \cP^{ij} \partial_{r} \partial_{j} \cP^{k\ell} \partial_{\ell} \cP^{mn} \partial_{v} \partial_{n} \cP^{pq} \partial_{q} \cP^{rs} \partial_{s} \cP^{tv} \partial_{i} f \partial_{k} g %\\
\end{align*}
%\\   
\begin{align*}
-5 \partial_{t} \partial_{n} \cP^{ij} \partial_{j} \cP^{k\ell} \partial_{r} \partial_{p} \partial_{k} \cP^{mn} \partial_{v} \partial_{\ell} \cP^{pq} \partial_{q} \cP^{rs} \partial_{s} \cP^{tv} \partial_{i} f \partial_{m} g
+5 \partial_{r} \partial_{p} \partial_{m} \cP^{ij} \partial_{t} \partial_{j} \cP^{k\ell} \partial_{\ell} \cP^{mn} \partial_{v} \partial_{n} \cP^{pq} \partial_{q} \cP^{rs} \partial_{s} \cP^{tv} \partial_{i} f \partial_{k} g \\
-5 \partial_{t} \partial_{p} \cP^{ij} \partial_{v} \partial_{r} \partial_{j} \cP^{k\ell} \partial_{\ell} \cP^{mn} \partial_{m} \cP^{pq} \partial_{q} \partial_{n} \cP^{rs} \partial_{s} \cP^{tv} \partial_{i} f \partial_{k} g
-5 \partial_{v} \partial_{r} \partial_{m} \cP^{ij} \partial_{j} \cP^{k\ell} \partial_{p} \partial_{k} \cP^{mn} \partial_{\ell} \cP^{pq} \partial_{n} \cP^{rs} \partial_{s} \partial_{q} \cP^{tv} \partial_{i} f \partial_{t} g \\
-5 \partial_{r} \partial_{p} \cP^{ij} \partial_{t} \partial_{q} \partial_{j} \cP^{k\ell} \partial_{v} \partial_{\ell} \cP^{mn} \partial_{m} \cP^{pq} \partial_{n} \cP^{rs} \partial_{s} \cP^{tv} \partial_{i} f \partial_{k} g
-5 \partial_{t} \partial_{s} \partial_{m} \cP^{ij} \partial_{v} \partial_{j} \cP^{k\ell} \partial_{k} \cP^{mn} \partial_{\ell} \cP^{pq} \partial_{p} \partial_{n} \cP^{rs} \partial_{q} \cP^{tv} \partial_{i} f \partial_{r} g \\
-5 \partial_{t} \partial_{p} \cP^{ij} \partial_{r} \partial_{q} \partial_{j} \cP^{k\ell} \partial_{v} \partial_{\ell} \cP^{mn} \partial_{m} \cP^{pq} \partial_{n} \cP^{rs} \partial_{s} \cP^{tv} \partial_{i} f \partial_{k} g
+5 \partial_{s} \partial_{p} \partial_{m} \cP^{ij} \partial_{t} \partial_{j} \cP^{k\ell} \partial_{k} \cP^{mn} \partial_{\ell} \cP^{pq} \partial_{v} \partial_{n} \cP^{rs} \partial_{q} \cP^{tv} \partial_{i} f \partial_{r} g \\
-5 \partial_{r} \partial_{m} \cP^{ij} \partial_{t} \partial_{p} \partial_{j} \cP^{k\ell} \partial_{\ell} \cP^{mn} \partial_{v} \partial_{n} \cP^{pq} \partial_{q} \cP^{rs} \partial_{s} \cP^{tv} \partial_{i} f \partial_{k} g
+5 \partial_{t} \partial_{p} \partial_{n} \cP^{ij} \partial_{j} \cP^{k\ell} \partial_{r} \partial_{k} \cP^{mn} \partial_{v} \partial_{\ell} \cP^{pq} \partial_{q} \cP^{rs} \partial_{s} \cP^{tv} \partial_{i} f \partial_{m} g \\
-5 \partial_{t} \partial_{m} \cP^{ij} \partial_{r} \partial_{p} \partial_{j} \cP^{k\ell} \partial_{\ell} \cP^{mn} \partial_{v} \partial_{n} \cP^{pq} \partial_{q} \cP^{rs} \partial_{s} \cP^{tv} \partial_{i} f \partial_{k} g
+5 \partial_{r} \partial_{p} \partial_{n} \cP^{ij} \partial_{j} \cP^{k\ell} \partial_{t} \partial_{k} \cP^{mn} \partial_{v} \partial_{\ell} \cP^{pq} \partial_{q} \cP^{rs} \partial_{s} \cP^{tv} \partial_{i} f \partial_{m} g \\
-5 \partial_{t} \partial_{s} \cP^{ij} \partial_{j} \cP^{k\ell} \partial_{k} \cP^{mn} \partial_{m} \partial_{\ell} \cP^{pq} \partial_{v} \partial_{p} \partial_{n} \cP^{rs} \partial_{q} \cP^{tv} \partial_{i} f \partial_{r} g
+5 \partial_{t} \partial_{r} \partial_{p} \cP^{ij} \partial_{v} \partial_{j} \cP^{k\ell} \partial_{\ell} \cP^{mn} \partial_{m} \cP^{pq} \partial_{q} \partial_{n} \cP^{rs} \partial_{s} \cP^{tv} \partial_{i} f \partial_{k} g \\
-5 \partial_{r} \partial_{p} \partial_{m} \partial_{k} \cP^{ij} \partial_{t} \cP^{k\ell} \partial_{v} \partial_{\ell} \cP^{mn} \partial_{n} \cP^{pq} \partial_{q} \cP^{rs} \partial_{s} \cP^{tv} \partial_{i} f \partial_{j} g
-5 \partial_{t} \partial_{p} \partial_{m} \partial_{k} \cP^{ij} \partial_{r} \cP^{k\ell} \partial_{\ell} \cP^{mn} \partial_{v} \partial_{n} \cP^{pq} \partial_{q} \cP^{rs} \partial_{s} \cP^{tv} \partial_{i} f \partial_{j} g \\
-5 \partial_{r} \partial_{p} \partial_{m} \partial_{k} \cP^{ij} \partial_{t} \cP^{k\ell} \partial_{\ell} \cP^{mn} \partial_{v} \partial_{n} \cP^{pq} \partial_{q} \cP^{rs} \partial_{s} \cP^{tv} \partial_{i} f \partial_{j} g
+5 \partial_{s} \partial_{m} \cP^{ij} \partial_{j} \cP^{k\ell} \partial_{t} \partial_{k} \cP^{mn} \partial_{\ell} \cP^{pq} \partial_{v} \partial_{p} \partial_{n} \cP^{rs} \partial_{q} \cP^{tv} \partial_{i} f \partial_{r} g \\
-5 \partial_{t} \partial_{r} \partial_{p} \cP^{ij} \partial_{q} \partial_{j} \cP^{k\ell} \partial_{\ell} \cP^{mn} \partial_{v} \partial_{m} \cP^{pq} \partial_{n} \cP^{rs} \partial_{s} \cP^{tv} \partial_{i} f \partial_{k} g
-5 \partial_{t} \partial_{n} \cP^{ij} \partial_{v} \partial_{j} \cP^{k\ell} \partial_{r} \partial_{p} \partial_{k} \cP^{mn} \partial_{\ell} \cP^{pq} \partial_{q} \cP^{rs} \partial_{s} \cP^{tv} \partial_{i} f \partial_{m} g \\
+5 \partial_{r} \partial_{p} \partial_{m} \cP^{ij} \partial_{t} \partial_{j} \cP^{k\ell} \partial_{v} \partial_{\ell} \cP^{mn} \partial_{n} \cP^{pq} \partial_{q} \cP^{rs} \partial_{s} \cP^{tv} \partial_{i} f \partial_{k} g
-5 \partial_{p} \partial_{n} \cP^{ij} \partial_{t} \partial_{j} \cP^{k\ell} \partial_{v} \partial_{r} \partial_{k} \cP^{mn} \partial_{\ell} \cP^{pq} \partial_{q} \cP^{rs} \partial_{s} \cP^{tv} \partial_{i} f \partial_{m} g \\
-5 \partial_{t} \partial_{r} \partial_{m} \cP^{ij} \partial_{p} \partial_{j} \cP^{k\ell} \partial_{v} \partial_{\ell} \cP^{mn} \partial_{n} \cP^{pq} \partial_{q} \cP^{rs} \partial_{s} \cP^{tv} \partial_{i} f \partial_{k} g
-5 \partial_{t} \partial_{p} \cP^{ij} \partial_{r} \partial_{q} \partial_{j} \cP^{k\ell} \partial_{\ell} \cP^{mn} \partial_{v} \partial_{m} \cP^{pq} \partial_{n} \cP^{rs} \partial_{s} \cP^{tv} \partial_{i} f \partial_{k} g \\
+5 \partial_{s} \partial_{p} \partial_{m} \cP^{ij} \partial_{j} \cP^{k\ell} \partial_{t} \partial_{k} \cP^{mn} \partial_{\ell} \cP^{pq} \partial_{v} \partial_{n} \cP^{rs} \partial_{q} \cP^{tv} \partial_{i} f \partial_{r} g
-5 \partial_{t} \partial_{p} \cP^{ij} \partial_{v} \partial_{r} \partial_{j} \cP^{k\ell} \partial_{\ell} \cP^{mn} \partial_{m} \cP^{pq} \partial_{n} \cP^{rs} \partial_{s} \partial_{q} \cP^{tv} \partial_{i} f \partial_{k} g \\
-5 \partial_{v} \partial_{r} \partial_{m} \cP^{ij} \partial_{j} \cP^{k\ell} \partial_{k} \cP^{mn} \partial_{\ell} \cP^{pq} \partial_{p} \partial_{n} \cP^{rs} \partial_{s} \partial_{q} \cP^{tv} \partial_{i} f \partial_{t} g
-5 \partial_{t} \partial_{m} \cP^{ij} \partial_{r} \partial_{p} \partial_{j} \cP^{k\ell} \partial_{v} \partial_{\ell} \cP^{mn} \partial_{n} \cP^{pq} \partial_{q} \cP^{rs} \partial_{s} \cP^{tv} \partial_{i} f \partial_{k} g \\
-5 \partial_{r} \partial_{p} \partial_{n} \cP^{ij} \partial_{t} \partial_{j} \cP^{k\ell} \partial_{v} \partial_{k} \cP^{mn} \partial_{\ell} \cP^{pq} \partial_{q} \cP^{rs} \partial_{s} \cP^{tv} \partial_{i} f \partial_{m} g
+5 \partial_{p} \partial_{m} \cP^{ij} \partial_{t} \partial_{r} \partial_{j} \cP^{k\ell} \partial_{v} \partial_{\ell} \cP^{mn} \partial_{n} \cP^{pq} \partial_{q} \cP^{rs} \partial_{s} \cP^{tv} \partial_{i} f \partial_{k} g \\
-5 \partial_{t} \partial_{r} \partial_{n} \cP^{ij} \partial_{v} \partial_{j} \cP^{k\ell} \partial_{p} \partial_{k} \cP^{mn} \partial_{\ell} \cP^{pq} \partial_{q} \cP^{rs} \partial_{s} \cP^{tv} \partial_{i} f \partial_{m} g
-5 \partial_{t} \partial_{p} \cP^{ij} \partial_{j} \cP^{k\ell} \partial_{v} \partial_{k} \cP^{mn} \partial_{r} \partial_{\ell} \cP^{pq} \partial_{n} \cP^{rs} \partial_{s} \partial_{q} \cP^{tv} \partial_{i} f \partial_{m} g \\
-5 \partial_{t} \partial_{m} \cP^{ij} \partial_{j} \cP^{k\ell} \partial_{p} \partial_{k} \cP^{mn} \partial_{v} \partial_{\ell} \cP^{pq} \partial_{q} \partial_{n} \cP^{rs} \partial_{s} \cP^{tv} \partial_{i} f \partial_{r} g
+5 \partial_{r} \partial_{p} \cP^{ij} \partial_{j} \cP^{k\ell} \partial_{t} \partial_{k} \cP^{mn} \partial_{s} \partial_{\ell} \cP^{pq} \partial_{v} \partial_{n} \cP^{rs} \partial_{q} \cP^{tv} \partial_{i} f \partial_{m} g \\
-5 \partial_{t} \partial_{r} \cP^{ij} \partial_{v} \partial_{j} \cP^{k\ell} \partial_{p} \partial_{k} \cP^{mn} \partial_{s} \partial_{\ell} \cP^{pq} \partial_{n} \cP^{rs} \partial_{q} \cP^{tv} \partial_{i} f \partial_{m} g 
+5 \partial_{r} \partial_{p} \cP^{ij} \partial_{j} \cP^{k\ell} \partial_{t} \partial_{k} \cP^{mn} \partial_{v} \partial_{\ell} \cP^{pq} \partial_{q} \partial_{n} \cP^{rs} \partial_{s} \cP^{tv} \partial_{i} f \partial_{m} g \\
+5 \partial_{r} \partial_{p} \cP^{ij} \partial_{t} \partial_{j} \cP^{k\ell} \partial_{v} \partial_{k} \cP^{mn} \partial_{\ell} \cP^{pq} \partial_{n} \cP^{rs} \partial_{s} \partial_{q} \cP^{tv} \partial_{i} f \partial_{m} g
-5 \partial_{t} \partial_{m} \cP^{ij} \partial_{v} \partial_{j} \cP^{k\ell} \partial_{k} \cP^{mn} \partial_{\ell} \cP^{pq} \partial_{p} \partial_{n} \cP^{rs} \partial_{s} \partial_{q} \cP^{tv} \partial_{i} f \partial_{r} g \\
+5 \partial_{t} \partial_{r} \cP^{ij} \partial_{s} \partial_{j} \cP^{k\ell} \partial_{v} \partial_{k} \cP^{mn} \partial_{n} \partial_{\ell} \cP^{pq} \partial_{p} \cP^{rs} \partial_{q} \cP^{tv} \partial_{i} f \partial_{m} g
-5 \partial_{t} \partial_{m} \cP^{ij} \partial_{v} \partial_{j} \cP^{k\ell} \partial_{p} \partial_{k} \cP^{mn} \partial_{\ell} \cP^{pq} \partial_{q} \partial_{n} \cP^{rs} \partial_{s} \cP^{tv} \partial_{i} f \partial_{r} g \\
+5 \partial_{t} \partial_{r} \cP^{ij} \partial_{j} \cP^{k\ell} \partial_{s} \partial_{k} \cP^{mn} \partial_{n} \partial_{\ell} \cP^{pq} \partial_{v} \partial_{p} \cP^{rs} \partial_{q} \cP^{tv} \partial_{i} f \partial_{m} g
+5 \partial_{p} \partial_{m} \cP^{ij} \partial_{t} \partial_{j} \cP^{k\ell} \partial_{k} \cP^{mn} \partial_{s} \partial_{\ell} \cP^{pq} \partial_{v} \partial_{n} \cP^{rs} \partial_{q} \cP^{tv} \partial_{i} f \partial_{r} g \\
+5 \partial_{t} \partial_{p} \cP^{ij} \partial_{r} \partial_{j} \cP^{k\ell} \partial_{v} \partial_{k} \cP^{mn} \partial_{\ell} \cP^{pq} \partial_{q} \partial_{n} \cP^{rs} \partial_{s} \cP^{tv} \partial_{i} f \partial_{m} g
-5 \partial_{p} \partial_{m} \cP^{ij} \partial_{t} \partial_{j} \cP^{k\ell} \partial_{v} \partial_{k} \cP^{mn} \partial_{\ell} \cP^{pq} \partial_{q} \partial_{n} \cP^{rs} \partial_{s} \cP^{tv} \partial_{i} f \partial_{r} g \\
-5 \partial_{t} \partial_{m} \cP^{ij} \partial_{j} \cP^{k\ell} \partial_{p} \partial_{k} \cP^{mn} \partial_{s} \partial_{\ell} \cP^{pq} \partial_{v} \partial_{n} \cP^{rs} \partial_{q} \cP^{tv} \partial_{i} f \partial_{r} g
+5 \partial_{r} \partial_{p} \cP^{ij} \partial_{t} \partial_{j} \cP^{k\ell} \partial_{v} \partial_{k} \cP^{mn} \partial_{s} \partial_{\ell} \cP^{pq} \partial_{n} \cP^{rs} \partial_{q} \cP^{tv} \partial_{i} f \partial_{m} g \\
+5 \partial_{t} \partial_{r} \cP^{ij} \partial_{j} \cP^{k\ell} \partial_{p} \partial_{k} \cP^{mn} \partial_{s} \partial_{\ell} \cP^{pq} \partial_{v} \partial_{n} \cP^{rs} \partial_{q} \cP^{tv} \partial_{i} f \partial_{m} g
-5 \partial_{t} \partial_{r} \cP^{ij} \partial_{v} \partial_{j} \cP^{k\ell} \partial_{p} \partial_{k} \cP^{mn} \partial_{\ell} \cP^{pq} \partial_{n} \cP^{rs} \partial_{s} \partial_{q} \cP^{tv} \partial_{i} f \partial_{m} g \\
+5 \partial_{t} \partial_{r} \cP^{ij} \partial_{j} \cP^{k\ell} \partial_{p} \partial_{k} \cP^{mn} \partial_{v} \partial_{\ell} \cP^{pq} \partial_{q} \partial_{n} \cP^{rs} \partial_{s} \cP^{tv} \partial_{i} f \partial_{m} g
-5 \partial_{t} \partial_{p} \cP^{ij} \partial_{r} \partial_{j} \cP^{k\ell} \partial_{v} \partial_{k} \cP^{mn} \partial_{s} \partial_{\ell} \cP^{pq} \partial_{n} \cP^{rs} \partial_{q} \cP^{tv} \partial_{i} f \partial_{m} g \\
-5 \partial_{t} \partial_{m} \cP^{ij} \partial_{j} \cP^{k\ell} \partial_{s} \partial_{k} \cP^{mn} \partial_{n} \partial_{\ell} \cP^{pq} \partial_{v} \partial_{p} \cP^{rs} \partial_{q} \cP^{tv} \partial_{i} f \partial_{r} g
-5 \partial_{r} \partial_{m} \cP^{ij} \partial_{v} \partial_{j} \cP^{k\ell} \partial_{p} \partial_{k} \cP^{mn} \partial_{\ell} \cP^{pq} \partial_{n} \cP^{rs} \partial_{s} \partial_{q} \cP^{tv} \partial_{i} f \partial_{t} g \\
-5 \partial_{t} \partial_{m} \cP^{ij} \partial_{s} \partial_{j} \cP^{k\ell} \partial_{k} \cP^{mn} \partial_{n} \partial_{\ell} \cP^{pq} \partial_{v} \partial_{p} \cP^{rs} \partial_{q} \cP^{tv} \partial_{i} f \partial_{r} g
+5 \partial_{t} \partial_{m} \cP^{ij} \partial_{j} \cP^{k\ell} \partial_{v} \partial_{k} \cP^{mn} \partial_{\ell} \cP^{pq} \partial_{p} \partial_{n} \cP^{rs} \partial_{s} \partial_{q} \cP^{tv} \partial_{i} f \partial_{r} g \\
+5 \partial_{t} \partial_{p} \cP^{ij} \partial_{v} \partial_{j} \cP^{k\ell} \partial_{q} \partial_{k} \cP^{mn} \partial_{r} \partial_{\ell} \cP^{pq} \partial_{n} \cP^{rs} \partial_{s} \cP^{tv} \partial_{i} f \partial_{m} g
-5 \partial_{r} \partial_{m} \cP^{ij} \partial_{j} \cP^{k\ell} \partial_{v} \partial_{k} \cP^{mn} \partial_{\ell} \cP^{pq} \partial_{p} \partial_{n} \cP^{rs} \partial_{s} \partial_{q} \cP^{tv} \partial_{i} f \partial_{t} g \\
+5 \partial_{t} \partial_{p} \cP^{ij} \partial_{r} \partial_{j} \cP^{k\ell} \partial_{q} \partial_{k} \cP^{mn} \partial_{v} \partial_{\ell} \cP^{pq} \partial_{n} \cP^{rs} \partial_{s} \cP^{tv} \partial_{i} f \partial_{m} g
+5 \partial_{p} \partial_{m} \cP^{ij} \partial_{s} \partial_{j} \cP^{k\ell} \partial_{t} \partial_{k} \cP^{mn} \partial_{\ell} \cP^{pq} \partial_{v} \partial_{n} \cP^{rs} \partial_{q} \cP^{tv} \partial_{i} f \partial_{r} g \\
-5 \partial_{t} \partial_{m} \cP^{ij} \partial_{s} \partial_{j} \cP^{k\ell} \partial_{v} \partial_{k} \cP^{mn} \partial_{\ell} \cP^{pq} \partial_{p} \partial_{n} \cP^{rs} \partial_{q} \cP^{tv} \partial_{i} f \partial_{r} g
-5 \partial_{r} \partial_{p} \cP^{ij} \partial_{s} \partial_{j} \cP^{k\ell} \partial_{t} \partial_{k} \cP^{mn} \partial_{\ell} \cP^{pq} \partial_{v} \partial_{n} \cP^{rs} \partial_{q} \cP^{tv} \partial_{i} f \partial_{m} g \\
+5 \partial_{t} \partial_{p} \cP^{ij} \partial_{v} \partial_{j} \cP^{k\ell} \partial_{r} \partial_{k} \cP^{mn} \partial_{n} \partial_{\ell} \cP^{pq} \partial_{q} \cP^{rs} \partial_{s} \cP^{tv} \partial_{i} f \partial_{m} g
-5 \partial_{r} \partial_{p} \cP^{ij} \partial_{t} \partial_{j} \cP^{k\ell} \partial_{v} \partial_{k} \cP^{mn} \partial_{n} \partial_{\ell} \cP^{pq} \partial_{q} \cP^{rs} \partial_{s} \cP^{tv} \partial_{i} f \partial_{m} g \\
-5 \partial_{t} \partial_{r} \cP^{ij} \partial_{s} \partial_{j} \cP^{k\ell} \partial_{p} \partial_{k} \cP^{mn} \partial_{\ell} \cP^{pq} \partial_{v} \partial_{n} \cP^{rs} \partial_{q} \cP^{tv} \partial_{i} f \partial_{m} g
-5 \partial_{r} \partial_{p} \cP^{ij} \partial_{j} \cP^{k\ell} \partial_{t} \partial_{q} \partial_{k} \cP^{mn} \partial_{v} \partial_{\ell} \cP^{pq} \partial_{n} \cP^{rs} \partial_{s} \cP^{tv} \partial_{i} f \partial_{m} g \\
-5 \partial_{t} \partial_{r} \partial_{p} \cP^{ij} \partial_{j} \cP^{k\ell} \partial_{v} \partial_{k} \cP^{mn} \partial_{\ell} \cP^{pq} \partial_{n} \cP^{rs} \partial_{s} \partial_{q} \cP^{tv} \partial_{i} f \partial_{m} g
+5 \partial_{t} \partial_{p} \cP^{ij} \partial_{j} \cP^{k\ell} \partial_{v} \partial_{q} \partial_{k} \cP^{mn} \partial_{r} \partial_{\ell} \cP^{pq} \partial_{n} \cP^{rs} \partial_{s} \cP^{tv} \partial_{i} f \partial_{m} g \\
-5 \partial_{t} \partial_{p} \partial_{m} \cP^{ij} \partial_{j} \cP^{k\ell} \partial_{v} \partial_{k} \cP^{mn} \partial_{\ell} \cP^{pq} \partial_{q} \partial_{n} \cP^{rs} \partial_{s} \cP^{tv} \partial_{i} f \partial_{r} g
+5 \partial_{t} \partial_{p} \cP^{ij} \partial_{j} \cP^{k\ell} \partial_{v} \partial_{r} \partial_{k} \cP^{mn} \partial_{n} \partial_{\ell} \cP^{pq} \partial_{q} \cP^{rs} \partial_{s} \cP^{tv} \partial_{i} f \partial_{m} g \\
+5 \partial_{t} \partial_{r} \partial_{p} \cP^{ij} \partial_{s} \partial_{j} \cP^{k\ell} \partial_{v} \partial_{k} \cP^{mn} \partial_{\ell} \cP^{pq} \partial_{n} \cP^{rs} \partial_{q} \cP^{tv} \partial_{i} f \partial_{m} g 
+5 \partial_{r} \partial_{p} \cP^{ij} \partial_{t} \partial_{j} \cP^{k\ell} \partial_{v} \partial_{q} \partial_{k} \cP^{mn} \partial_{\ell} \cP^{pq} \partial_{n} \cP^{rs} \partial_{s} \cP^{tv} \partial_{i} f \partial_{m} g \\
-5 \partial_{t} \partial_{r} \partial_{p} \cP^{ij} \partial_{j} \cP^{k\ell} \partial_{v} \partial_{k} \cP^{mn} \partial_{\ell} \cP^{pq} \partial_{q} \partial_{n} \cP^{rs} \partial_{s} \cP^{tv} \partial_{i} f \partial_{m} g
+5 \partial_{t} \partial_{p} \cP^{ij} \partial_{r} \partial_{j} \cP^{k\ell} \partial_{v} \partial_{q} \partial_{k} \cP^{mn} \partial_{\ell} \cP^{pq} \partial_{n} \cP^{rs} \partial_{s} \cP^{tv} \partial_{i} f \partial_{m} g \\
+5 \partial_{t} \partial_{p} \partial_{m} \cP^{ij} \partial_{j} \cP^{k\ell} \partial_{k} \cP^{mn} \partial_{s} \partial_{\ell} \cP^{pq} \partial_{v} \partial_{n} \cP^{rs} \partial_{q} \cP^{tv} \partial_{i} f \partial_{r} g
-5 \partial_{t} \partial_{m} \cP^{ij} \partial_{s} \partial_{j} \cP^{k\ell} \partial_{k} \cP^{mn} \partial_{\ell} \cP^{pq} \partial_{v} \partial_{p} \partial_{n} \cP^{rs} \partial_{q} \cP^{tv} \partial_{i} f \partial_{r} g \\
+5 \partial_{t} \partial_{p} \partial_{m} \cP^{ij} \partial_{s} \partial_{j} \cP^{k\ell} \partial_{k} \cP^{mn} \partial_{\ell} \cP^{pq} \partial_{v} \partial_{n} \cP^{rs} \partial_{q} \cP^{tv} \partial_{i} f \partial_{r} g
-5 \partial_{t} \partial_{r} \cP^{ij} \partial_{s} \partial_{j} \cP^{k\ell} \partial_{v} \partial_{p} \partial_{k} \cP^{mn} \partial_{\ell} \cP^{pq} \partial_{n} \cP^{rs} \partial_{q} \cP^{tv} \partial_{i} f \partial_{m} g \\
-5 \partial_{t} \partial_{r} \partial_{p} \cP^{ij} \partial_{j} \cP^{k\ell} \partial_{v} \partial_{k} \cP^{mn} \partial_{n} \partial_{\ell} \cP^{pq} \partial_{q} \cP^{rs} \partial_{s} \cP^{tv} \partial_{i} f \partial_{m} g
+5 \partial_{t} \partial_{r} \cP^{ij} \partial_{j} \cP^{k\ell} \partial_{v} \partial_{p} \partial_{k} \cP^{mn} \partial_{\ell} \cP^{pq} \partial_{n} \cP^{rs} \partial_{s} \partial_{q} \cP^{tv} \partial_{i} f \partial_{m} g \\
+5 \partial_{t} \partial_{r} \partial_{p} \cP^{ij} \partial_{j} \cP^{k\ell} \partial_{q} \partial_{k} \cP^{mn} \partial_{v} \partial_{\ell} \cP^{pq} \partial_{n} \cP^{rs} \partial_{s} \cP^{tv} \partial_{i} f \partial_{m} g
+5 \partial_{t} \partial_{r} \cP^{ij} \partial_{j} \cP^{k\ell} \partial_{v} \partial_{p} \partial_{k} \cP^{mn} \partial_{\ell} \cP^{pq} \partial_{q} \partial_{n} \cP^{rs} \partial_{s} \cP^{tv} \partial_{i} f \partial_{m} g \\
+5 \partial_{t} \partial_{r} \partial_{p} \cP^{ij} \partial_{v} \partial_{j} \cP^{k\ell} \partial_{q} \partial_{k} \cP^{mn} \partial_{\ell} \cP^{pq} \partial_{n} \cP^{rs} \partial_{s} \cP^{tv} \partial_{i} f \partial_{m} g
+5 \partial_{v} \partial_{m} \cP^{ij} \partial_{j} \cP^{k\ell} \partial_{p} \partial_{k} \cP^{mn} \partial_{r} \partial_{\ell} \cP^{pq} \partial_{n} \cP^{rs} \partial_{s} \partial_{q} \cP^{tv} \partial_{i} f \partial_{t} g \\
+5 \partial_{t} \partial_{p} \cP^{ij} \partial_{r} \partial_{j} \cP^{k\ell} \partial_{\ell} \cP^{mn} \partial_{v} \partial_{m} \cP^{pq} \partial_{q} \partial_{n} \cP^{rs} \partial_{s} \cP^{tv} \partial_{i} f \partial_{k} g
-5 \partial_{s} \partial_{m} \cP^{ij} \partial_{j} \cP^{k\ell} \partial_{t} \partial_{k} \cP^{mn} \partial_{n} \partial_{\ell} \cP^{pq} \partial_{v} \partial_{p} \cP^{rs} \partial_{q} \cP^{tv} \partial_{i} f \partial_{r} g %\\
\end{align*}
\begin{align*}
+5 \partial_{t} \partial_{p} \cP^{ij} \partial_{r} \partial_{j} \cP^{k\ell} \partial_{\ell} \cP^{mn} \partial_{s} \partial_{m} \cP^{pq} \partial_{v} \partial_{n} \cP^{rs} \partial_{q} \cP^{tv} \partial_{i} f \partial_{k} g
-5 \partial_{s} \partial_{m} \cP^{ij} \partial_{t} \partial_{j} \cP^{k\ell} \partial_{k} \cP^{mn} \partial_{n} \partial_{\ell} \cP^{pq} \partial_{v} \partial_{p} \cP^{rs} \partial_{q} \cP^{tv} \partial_{i} f \partial_{r} g \\
+5 \partial_{t} \partial_{p} \cP^{ij} \partial_{r} \partial_{j} \cP^{k\ell} \partial_{v} \partial_{\ell} \cP^{mn} \partial_{s} \partial_{m} \cP^{pq} \partial_{n} \cP^{rs} \partial_{q} \cP^{tv} \partial_{i} f \partial_{k} g
+5 \partial_{v} \partial_{m} \cP^{ij} \partial_{r} \partial_{j} \cP^{k\ell} \partial_{k} \cP^{mn} \partial_{\ell} \cP^{pq} \partial_{p} \partial_{n} \cP^{rs} \partial_{s} \partial_{q} \cP^{tv} \partial_{i} f \partial_{t} g \\
+5 \partial_{t} \partial_{p} \cP^{ij} \partial_{r} \partial_{j} \cP^{k\ell} \partial_{v} \partial_{\ell} \cP^{mn} \partial_{m} \cP^{pq} \partial_{n} \cP^{rs} \partial_{s} \partial_{q} \cP^{tv} \partial_{i} f \partial_{k} g
+5 \partial_{t} \partial_{s} \cP^{ij} \partial_{v} \partial_{j} \cP^{k\ell} \partial_{k} \cP^{mn} \partial_{m} \partial_{\ell} \cP^{pq} \partial_{p} \partial_{n} \cP^{rs} \partial_{q} \cP^{tv} \partial_{i} f \partial_{r} g \\
+5 \partial_{r} \partial_{p} \cP^{ij} \partial_{t} \partial_{j} \cP^{k\ell} \partial_{v} \partial_{\ell} \cP^{mn} \partial_{m} \cP^{pq} \partial_{q} \partial_{n} \cP^{rs} \partial_{s} \cP^{tv} \partial_{i} f \partial_{k} g
-5 \partial_{t} \partial_{p} \partial_{n} \cP^{ij} \partial_{j} \cP^{k\ell} \partial_{v} \partial_{r} \partial_{k} \cP^{mn} \partial_{\ell} \cP^{pq} \partial_{q} \cP^{rs} \partial_{s} \cP^{tv} \partial_{i} f \partial_{m} g \\
-5 \partial_{t} \partial_{r} \partial_{m} \cP^{ij} \partial_{v} \partial_{p} \partial_{j} \cP^{k\ell} \partial_{\ell} \cP^{mn} \partial_{n} \cP^{pq} \partial_{q} \cP^{rs} \partial_{s} \cP^{tv} \partial_{i} f \partial_{k} g
+5 \partial_{t} \partial_{s} \partial_{m} \cP^{ij} \partial_{j} \cP^{k\ell} \partial_{k} \cP^{mn} \partial_{\ell} \cP^{pq} \partial_{v} \partial_{p} \partial_{n} \cP^{rs} \partial_{q} \cP^{tv} \partial_{i} f \partial_{r} g \\
-5 \partial_{t} \partial_{r} \partial_{p} \cP^{ij} \partial_{v} \partial_{q} \partial_{j} \cP^{k\ell} \partial_{\ell} \cP^{mn} \partial_{m} \cP^{pq} \partial_{n} \cP^{rs} \partial_{s} \cP^{tv} \partial_{i} f \partial_{k} g
-5 \partial_{t} \partial_{r} \partial_{n} \cP^{ij} \partial_{j} \cP^{k\ell} \partial_{v} \partial_{p} \partial_{k} \cP^{mn} \partial_{\ell} \cP^{pq} \partial_{q} \cP^{rs} \partial_{s} \cP^{tv} \partial_{i} f \partial_{m} g \\
-5 \partial_{t} \partial_{p} \partial_{m} \cP^{ij} \partial_{v} \partial_{r} \partial_{j} \cP^{k\ell} \partial_{\ell} \cP^{mn} \partial_{n} \cP^{pq} \partial_{q} \cP^{rs} \partial_{s} \cP^{tv} \partial_{i} f \partial_{k} g\lefteqn{.}
\end{align*}

}%\par=empty line to have right inter-line distances!

\end{document}